\documentclass[12pt]{article}
\usepackage{graphicx}
\usepackage{amsmath}
\usepackage{amssymb}
\usepackage{theorem}

\numberwithin{equation}{section}

\newtheorem{thm}{Theorem}[section]
\newtheorem{lemma}[thm]{Lemma}
\newtheorem{prop}[thm]{Proposition}
\newtheorem{cor}[thm]{Corollary}
{\theorembodyfont{\rmfamily}
\newtheorem{defn}[thm]{Definition}

\newtheorem{rmk}[thm]{Remark}
}

\newcommand{\qed}{\hfill \mbox{\raggedright \rule{.07in}{.1in}}}

\newenvironment{proof}{\vspace{1ex}\noindent{\bf
Proof}\hspace{0.5em}}{\hfill\qed\vspace{1ex}}

\newcommand{\R}{{\mathbb R}}

\newcommand{\Z}{{\mathbb Z}}

\newcommand{\N}{{\mathbb N}}
\newcommand{\E}{{\mathbb E}}
\newcommand{\BBW}{{\mathbb W}}

\newcommand{\cB}{{\mathcal B}}
\newcommand{\fB}{{\mathfrak B}}
\newcommand{\cW}{{\mathcal W}}

\newcommand{\overbar}[1]{\mkern 3.7mu\overline{\mkern-3.7mu#1\mkern-0.5mu}\mkern 0.5mu}
\newcommand{\overbaro}[1]{\mkern 1.5mu\overline{\mkern-1.5mu#1\mkern-1.5mu}\mkern 1.5mu}

\newcommand{\overbarr}[1]{\mkern 3.5mu\overline{\mkern-3.5mu#1\mkern-0.5mu}\mkern 0.5mu}

\newcommand{\bH}{{\overbar{H}}}

\newcommand{\bX}{{\overbaro X}}
\newcommand{\bY}{{\overbaro Y}}

\newcommand{\bm}{{\overbaro m}}
\newcommand{\bF}{\overbarr F}
\newcommand{\bmu}{\bar\mu}
\newcommand{\bpi}{\bar\pi}
\newcommand{\bV}{\overbaro V}

\newcommand{\hY}{{\widehat Y}}
\newcommand{\hV}{{\widehat V}}
\newcommand{\tV}{{\widetilde V}}
\newcommand{\tchi}{{\widetilde \chi}}
\newcommand{\tm}{{\widetilde m}}

\newcommand{\eps}{\epsilon}

\newcommand{\Vol}{\operatorname{Vol}}
\newcommand{\sgn}{\operatorname{sgn}}

\newcommand{\SMALL}{\textstyle}
\newcommand{\BIG}{\displaystyle}

\title{Statistical properties for flows with unbounded roof function, including
the Lorenz attractor}

\author{
P\'eter B\'alint \thanks{
MTA-BME Stochastics Research Group, Budapest University of Technology and Economics, Egry J\'ozsef
u. 1, H-1111 Budapest, Hungary and Institute of Mathematics,
Budapest University of Technology and Economics,
H-1111 Egry J\'ozsef u. 1
Budapest, Hungary.
pet@math.bme.hu}
\and
Ian Melbourne \thanks{
Mathematics Institute, University of Warwick, Coventry, CV4 7AL, UK.
i.melbourne@warwick.ac.uk}
}

\date{5 March 2018.  Revised 23 June 2018}

\begin{document}

\maketitle

\begin{abstract}
For geometric
Lorenz attractors (including the classical Lorenz attractor)
we obtain a greatly simplified proof of the central limit theorem
which applies also to
the more general class of codimension two singular hyperbolic attractors.

We also obtain the functional central limit theorem and moment estimates, as well as iterated versions of these results.  A consequence is deterministic homogenisation (convergence to a stochastic differential equation) for fast-slow dynamical systems whenever the fast dynamics is singularly hyperbolic of codimension two.
\end{abstract}

 \section{Introduction}
 \label{sec:intro}

The classical Lorenz attractor for the Lorenz
equations~\cite{Lorenz63},
\begin{align*}
\dot x = 10(y - x), \quad
\dot y = 28x -y -xz, \quad
\dot z = xy - \textstyle{\frac83}z,
\end{align*}
is a prototypical example of a chaotic dynamical system.
There has been much interest in studying rigorously the properties of geometric Lorenz attractors~\cite{AfraimovicBykovSilnikov77,GuckenWilliams79}
 which resemble the numerically-observed classical Lorenz attractor.
A computer-assisted proof that the attractor for the classical Lorenz equations is indeed a geometric Lorenz attractor was subsequently provided by~\cite{Tucker02}.
In contrast to Axiom~A attractors~\cite{Smale67}, the Lorenz attractor contains an equilibrium (steady-state for the flow) and cannot be structurally stable, though it is robustly nontrivial and transitive~\cite{Tucker02}.
Moreover, the Lorenz attractor is {\em singular hyperbolic}~\cite{MoralesPacificoPujals04} and exhibits
strong statistical properties such as expansivity~\cite{APPV09}, the central limit theorem (CLT)~\cite{HM07}, moment estimates~\cite{MTorok12},
and exponential decay of correlations~\cite{AraujoM16}.

Recently,~\cite{OvsyannikovTuraev17} (see also previous work of~\cite{DumortierKokubuOka95}) gave an analytic proof of existence of geometric Lorenz attractors in the
extended Lorenz model
\[
\dot x = y, \quad \dot y=-\lambda y+\gamma x(1-z)-\delta x^3,
\quad \dot z=-\alpha z+\beta x^2.
\]
A consequence of the current paper, in conjunction with~\cite{AraujoMsub}, is that
CLTs and moment estimates hold for these attractors.
Exponential decay of correlations remains an open question: the proof in~\cite{AraujoM16} relies on the existence of a smooth stable foliation for the flow, which holds for the classical Lorenz attractor~\cite{AraujoM17} but not for the examples of~\cite{DumortierKokubuOka95,OvsyannikovTuraev17}.
(As far as the results on the CLT and moment estimates go, our results for the extended Lorenz model are completely analytic, whereas results for the classical Lorenz attractor rely on the
computer-assisted proof in~\cite{Tucker02}.)

More generally, our results apply to all singular hyperbolic attractors for three-dimensional flows.
By~\cite{MoralesPacificoPujals04}, this incorporates all nontrivial robustly transitive attractors for three-dimensional flows (including nontrivial Axiom~A attractors as well as classical and geometric Lorenz attractors).
Our results apply
also to
  codimension two singular hyperbolic attractors in arbitrary dimension.

A standard step in the analysis of Lorenz attractors and singular hyperbolic attractors is to consider a suitable Poincar\'e map with good nonuniform hyperbolicity properties.   The return time function (roof function) is generally unbounded due to the presence of steady-states for the flow.
Due to this unboundedness, the
original proof of the CLT for the classical Lorenz equations~\cite{HM07} is quite complicated, relying on an inducing scheme that involves phase-space exclusion arguments of the type in~\cite{BenedicksCarleson85} to remove points that return too quickly to a vicinity of the steady-states.

As we pointed out in an earlier unpublished preprint version of this paper, it is possible to give a much simpler proof of the CLT than the one in~\cite{HM07}.  In this paper, we provide the details.   As in~\cite{HM07},
we prove also the functional version, namely the weak invariance principle (WIP).
We also provide moment estimates as advertised in~\cite[Section~5.3]{MTorok12} and our results apply to the general class of singular hyperbolic attractors in~\cite{AraujoMsub} (it is not clear that the argument in~\cite{HM07} applies in this generality).
In addition, we prove iterated versions of the WIP and
moment estimates as advertised in~\cite[Section~10.3]{KM16}.  By~\cite{KM16,KM17} it is therefore possible to prove homogenisation results where fast-slow systems converge to a limiting stochastic differential equation.

\begin{rmk}  One version of an earlier preprint of this paper was called
``Decay of correlations for flows with unbounded roof function, including the
infinite horizon planar periodic Lorentz gas''.  The correlation decay aspects of that preprint were incorrect and are subsumed by~\cite{BBMsub}. The remainder of that preprint is the basis of the current paper.

We note also that a ``rearrangement'' argument plays a much larger role in the older versions, but is still useful for optimal moment control when the $L^p$ class of the roof function is restricted, see Section~\ref{sec:rearrange} and Remark~\ref{rmk:rearrange}(ii).  (Rearrangement turns out not to be required for singular hyperbolic attractors where the roof function lies in $L^p$ for all $p<\infty$.)
\end{rmk}

\subsection{Consequences for singular hyperbolic attractors}

A specific example to which our main results apply is the classical Lorenz attractor~\cite{Lorenz63,Tucker02} (or more generally, the class of geometric Lorenz
attractors~\cite{AfraimovicBykovSilnikov77,GuckenWilliams79}).
More generally still, we consider singular hyperbolic attractors~\cite{MoralesPacificoPujals04}.

\begin{defn}
\label{def:SH}
Let $T_t:\R^n\to\R^n$ be a $C^r$ flow where $n\ge3$, $r>1$.
A compact invariant set
$\Lambda$ is {\em singular hyperbolic}
  if 

\noindent (i) All steady-states in $\Lambda$ are hyperbolic.
\\
\noindent (ii) The tangent bundle over $\Lambda$ can be
written as a continuous $DT_t$-invariant sum
$T_\Lambda M=E^s\oplus E^{cu}$,
and there exist constants $C>0$, $\lambda\in(0,1)$ such that
for all $x \in \Lambda$, $t\ge0$,
\begin{description}
\item[domination of the splitting:]
$\|DT_t | E^s_x\| \cdot \|DT_{-t} | E^{cu}_{T_tx}\| \le C \lambda^t$,
\item[uniform contraction along $E^s$:]
$\|DT_t | E^s_x\| \le C \lambda^t$,
\item[volume expansion along $E^{cu}$:]
$|\det(DT_t| E^{cu}_x)|\geq C^{-1}  \lambda^{-t}$.
\end{description}
Such an invariant set is {\em codimension two} if $\dim E^{cu}_x=2$ for $x\in\Lambda$.

A {\em singular hyperbolic attractor} is a singular hyperbolic set that is also an attractor (attracting and transitive).  We recall that $\Lambda$ is attracting if  there exists a positively invariant set $U\subset\R^n$ such that
$\Lambda=\bigcap_{t\ge0} T_tU$, and that $\Lambda$ is transitive if there exists a dense orbit in $\Lambda$.
\end{defn}

\begin{rmk}  (a) If $\Lambda$ is a codimension two singular hyperbolic attractor and
$x_0\in\Lambda$ is a steady-state, then $x_0$ is {\em Lorenz-like}: the linearized vector field at $x_0$ restricted to $E_{x_0}^{cu}$ has real eigenvalues $\lambda^s$, $\lambda^u$ satisfying $-\lambda^u<\lambda^s<0<\lambda^u$.
\\[.75ex]
(b)  
By~\cite[Section~8.1]{AraujoMsub}, an attracting codimension two singular hyperbolic attracting set admits a spectral decomposition into a finite disjoint union of singular hyperbolic attractors.
Hence, there is no loss of generality in assuming transitivity.
\\[.75ex]
(c)   As explained in Section~\ref{sec:Lorenz}, the results for singular hyperbolic attractors described below make use of~\cite{AraujoMsub} which relies heavily on the codimension two assumption.
Statistical properties for singular hyperbolic attractors with $\dim E^{cu}_x>2$ remains an interesting open question.
\end{rmk}

We now describe the main results of this paper in the setting of codimension two singular hyperbolic attractors $\Lambda\subset \R^n$.  More general statements for flows with unbounded roof function are given in the body of the paper.

\paragraph{Notation}
For $a,b\in\R^d$, we define the outer product
$a\otimes b=ab^T\in\R^{d\times d}$.
For real-valued functions $f,\,g$, the integral $\int f\,dg$ denotes
the It\^o integral (where defined).  Similarly, for vector-valued functions,
$\int f\otimes dg$ denotes matrices of It\^o integrals.
Also $\int f\otimes\circ dg$ denotes matrices of Stratonovich integrals.

Let $\eta\in(0,1]$.  Given $v:M\to\R^d$, where $M$ is a metric space, define $|v|_{\eta}=\sup_{x\neq x'}|v(x)-v(x')|/d(x,x')^\eta$ and $\|v\|_{\eta}=|v|_\infty+|v|_{\eta}$.
Let $C^\eta(M,\R^d)$ be the space of observables $v:M\to\R^d$ with $\|v\|_{\eta}<\infty$.  When $d=1$, we write $C^\eta(M)$.

Similarly, for $k\ge0$ an integer, we write 
$b\in C^{k,\eta}(\R^d\times M,\R^d)$ if $b=b(x,y)$ is $C^k$ in $x$ uniformly in $y$, and $C^\eta$ in $y$ uniformly in $x$.
We write $b\in C^{k+,\eta}(\R^d\times M,\R^d)$ if in addition there exists $\eta'\in(0,1)$ such that $\sup_y\|\partial b^k/\partial x^k(\cdot, y)\|_{\eta'}<\infty$.

If $\mu$ is a specified Borel probability measure on $M$, then we
define
$C_0^\eta(M,\R^d)=\{v\in C^\eta(M,\R^d):\int_M v\,d\mu=0\}$.

\vspace{2ex}

Let $T_t:\R^n\to\R^n$ denote a flow
with codimension two singular hyperbolic attractor~$\Lambda$.
By~\cite{AraujoMsub,APPV09,LeplaideurYang17}, there is a unique SRB (Sinai-Ruelle-Bowen) probability measure $\mu$ supported on $\Lambda$.
In particular, $\mu$ is a physical measure: there is an open set $U\subset\R^n$ and a subset $U'\subset U$ with $\Vol(U\setminus U')=0$ such that
$\lim_{t\to\infty}t^{-1}v_t(x_0)=\int_\Lambda v\,d\mu$
for every continuous function $v:\R^n\to\R$ and every initial condition $x_0\in U'$.
Here, $v_t=\int_0^t v\circ T_s\,ds$.  (In the case of the classical Lorenz attractor, $U=\R^3$.)  Fix $\eta\in(0,1]$.

\begin{thm}[CLT]  Let $v\in C_0^\eta(\Lambda)$.
Then the limit
$\sigma^2=\lim_{t\to\infty}t^{-1}\int_\Lambda v_t^2\,d\mu$
exists, and for all $c\in\R$,
\[
\lim_{t\to\infty}\mu(x_0\in\Lambda:t^{-1/2}v_t(x_0)\le c)=
(2\pi\sigma^2)^{-1/2}\int_{-\infty}^c e^{-y^2/(2\sigma^2)}\,dy.
\]
(That is, $t^{-1/2}v_t\to_d N(0,\sigma^2)$ as $t\to\infty$.)
Moreover, there is a closed subspace $S\subset C_0^\eta(\Lambda)$ of infinite codimension such that $\sigma^2>0$ for all $v\in C_0^\eta(\Lambda)\setminus S$.
\end{thm}

Next, let $W\in C[0,1]$ denote Brownian motion with variance $\sigma^2$ and
define
$W_n(t)=n^{-1/2}v_{nt}$.  We obtain the weak invariance principle:
\begin{description}
\item[(WIP)] $W_n\to_w W$ in $C[0,1]$ as $n\to\infty$,
for $v\in C_0^\eta(\Lambda)$.
\end{description}

Also, we obtain the following moment estimate:
\begin{description}
\item[(Moments)]
For any $p\in(0,\infty)$ there exists $C>0$ such that
$|v_t|_p\le Ct^{1/2}\|v\|_\eta$
for all $v\in C^\eta_0(\Lambda)$, $t>0$.
\end{description}

\begin{cor} \label{cor:conv}
 $\lim_{t\to\infty}\int_\Lambda |t^{-1/2}v_t|^p\,d\mu = \E |Y|^p$
for all $p\in(0,\infty)$, where $Y$ is normally distributed with mean zero and variance $\sigma^2$.
\end{cor}

\begin{proof}
This is a standard consequence of the CLT and moment estimates, see~\cite[Lemma~2.1(e)]{MTorok12}.~
\end{proof}

Next we consider iterated versions of the WIP and moment estimates.
Let $v\in C_0^\eta(\Lambda,\R^d)$ and define
$W_n:\Lambda\to C([0,1],\R^d)$ as before.  Also, define
\[
\BBW_n:\Lambda\to C([0,1],\R^{d\times d}),\qquad
\BBW_n(t)=\int_0^t W_n\otimes d W_n.
\]
(This is a standard Riemann integral since $W_n$ is $C^1$ in $t$.)

\begin{thm}[Iterated WIP]  \label{thm:WIP}
The limits
\[
\Sigma=\lim_{n\to\infty}\int_\Lambda W_n(1) \otimes W_n(1) \,d\mu,
\qquad E=\lim_{n\to\infty}\int_\Lambda \BBW_n(1)\,d\mu,
\]
exist and $(W_n,\BBW_n)\to_w(W,\BBW)$ in $C([0,1],\R^d\times\R^{d\times d})$,
where $W$ is $d$-dimensional Brownian motion with covariance $\Sigma$
and 
\[
\BBW(t)=
\int_0^t W\otimes dW+Et=
\int_0^t W\otimes\circ dW+Dt, \qquad D=E-\frac12\Sigma.
\]
Moreover, there is a closed subspace $S\subset C_0^\eta(\Lambda,\R^d)$ of infinite codimension such that $\det\Sigma>0$ for all $v\in C_0^\eta(\Lambda,\R^d)\setminus S$.
\end{thm}

\begin{rmk} \label{rmk:SigmaD}
The matrices $\Sigma,\,D\in \R^{d\times d}$ satisfy $\Sigma^T=\Sigma$
and $D^T=-D$.
If in addition the flow $T_t$ is mixing, then
\begin{align*}
\Sigma & \SMALL =\int_0^\infty \int_\Lambda\big(v\otimes(v\circ T_t)+(v\circ T_t)\otimes v\big)\,d\mu\, dt,
\\
D & \SMALL =\frac12\int_0^\infty \int_\Lambda\big(v\otimes(v\circ T_t)-(v\circ T_t)\otimes v\big)\,d\mu\, dt.
\end{align*}
\end{rmk}

Next, define
$S_t=\int_0^t \int_0^s (v\circ T_r)\otimes(v\circ T_s)\,dr\,ds$.
We obtain
\begin{description}
\item[(Iterated moments)]
For any $p\in(0,\infty)$ there exists $C>0$ such that
$|S_t|_p\le Ct\|v\|_\eta^2$
for all $v\in C^\eta_0(\Lambda,\R^d)$, $t>0$.
\end{description}

Finally, we mention the consequence of these results for deterministic homogenisation.
Consider a fast-slow system of the form
\[
\dot x  = a(x,y) + \eps^{-1}b(x,y) \qquad \dot y=\eps^{-2}g(y),
\]
with $x\in\R^d$ and $y\in\Lambda\subset\R^n$
where $\Lambda$ is a codimension two singular hyperbolic attractor for the ODE
$\dot y=g(y)$.  We suppose that $x(0)=\xi$ is fixed and that
$y(0)$ is chosen randomly from the probability space $(\Lambda,\mu)$.
Apart from the choice of initial condition $y(0)$, the system is completely deterministic.  

%

\begin{thm} Suppose that $a\in C^{1+,0}(\R^d\times\Lambda,\R^d)$ and $b\in C_0^{2+,\eta}(\R^d\times\Lambda,\R^d)$ for some $\eta\in(0,1]$.
For $v,w\in C_0^\eta(\Lambda)$, let $\fB(v,w)=\lim_{n\to\infty}n^{-1}\int_\Lambda \int_0^n\int_0^s v\circ T_r\,w\circ T_s\,dr\,ds\,d\mu$.
Define $\tilde a:\R^d\to\R^d$ and $\sigma:\R^d\to\R^{d\times d}$,
\begin{align*}
& \tilde a^i(x)  =\int_\Lambda a^i(x,y)\,d\mu(y)+\sum_{k=1}^d\fB(b^k(x,\cdot),\partial_kb^i(x,\cdot)), \\
& (\sigma(x)\sigma(x)^T)^{ij}=\fB(b^i(x,\cdot),b^j(x,\cdot))
+\fB(b^j(x,\cdot),b^i(x,\cdot)),
\end{align*}
for $i,j=1,\dots,d$.

Then $x_\eps\to_w X$ in $C([0,1],\R^d)$ where
$X$ is the unique solution of the SDE
\[
dX=\tilde a(X)\,dt+\sigma(X)\, dW, \quad X(0)=\xi.
\]
\end{thm}

\begin{proof}
We apply~\cite[Theorem~2.3]{KM17}.  
Define the bilinear operator
\[
\fB:C_0^\eta(\Lambda)\times C_0^\eta(\Lambda)\to\R,
\qquad
\fB(u,u')=\lim_{n\to\infty}\int_\Lambda \int_0^1 W_{u,n}\otimes dW_{u',n}\,d\mu,
\]
where $W_{u,n}=n^{-1/2}u_{nt}$.
By Theorem~\ref{thm:WIP}, for every $m\ge1$ and $v\in C_0^\eta(\Lambda,\R^m)$,
we have that
$(W_n,\BBW_n)\to_w (W,\BBW)$ in
$C([0,1],\R^m\times\R^{m\times m})$, and hence in the sense of finite-dimensional distributions, where $W$ is a Brownian motion and
$\BBW(t)=\int_0^t W\otimes dW+Et$.
Moreover,
\[
E^{ij}=\lim_{n\to\infty}\frac1n\int_\Lambda \BBW_n(1)^{ij}\,d\mu
=\lim_{n\to\infty}\frac1n\int_\Lambda \int_0^1 W_n^i\,dW_n^j\,d\mu
=\fB(v^i,v^j),
\]
for $i,j=1,\dots,m$.
This is Assumption~2.1 in~\cite{KM17}.
Also, 
Assumption~2.2 in~\cite{KM17} follows from the moment estimates 
$|v_t|_p\ll t^{1/2}$ and $|S_t|_p\ll t$ which hold for all $p$
and all $v\in C_0^\eta(\Lambda,\R^d)$.
Hence, the result follows from~\cite[Theorem~2.3]{KM17}.
\end{proof}

\begin{rmk}
The special case $a(x,y)=a(x)$ and $b(x,y)=b(x)v(y)$ 
with $a\in C^{1+}(\R^d,\R^d)$, $b\in C^{2+}(\R^d,\R^{d\times e})$,
$v\in C_0^\eta(\Lambda,\R^e)$, is considered 
in~\cite{KM16}.  For such fast-slow systems, the limiting SDE takes the form
\[
dX=\Big\{a(X)+\sum_{\alpha,\beta,\gamma=1}^d D^{\beta\gamma}\partial_\alpha b^\beta(X)b^{\alpha\gamma}(X)\Big\}dt+b(X)\circ dW, \quad X(0)=\xi,
\]
where $D$ is as in Theorem~\ref{thm:WIP}.
(Here $b^\beta(X)$ is the $\beta$'th column of $b(X)$.)
\end{rmk}

\vspace{1ex}
The remainder of this paper is organised as follows.
In Section~\ref{sec:NUE}, we prove the iterated WIP and moment estimates
for nonuniformly expanding systems with unbounded roof function.
In Section~\ref{sec:rearrange}, we use a rearrangement idea to obtain sharper moment estimates.
Sections~\ref{sec:exp} and~\ref{sec:prod} contain the iterated WIP and moment estimates for two classes of nonuniformly hyperbolic flows; one where there is exponential contraction along stable leaves for the flow, and one where there is a horseshoe structure as in~\cite{Young98}.
Finally, in Section~\ref{sec:Lorenz} we verify the hypotheses in Section~\ref{sec:exp} for singular hyperbolic attractors.  As a consequence of this, we obtain all of the results described in the introduction.

We end this introduction with some basic material on suspensions and inducing.

\subsection{Preliminaries on suspensions and inducing}
\label{sec:prelim}

\paragraph{Suspensions}
Let $T_t:M\to M$ be a semiflow defined on a measurable space $M$.
(We assume that $(x,t)\mapsto T_tx$ is measurable from $M\times[0,\infty)\to M$.)
Let $X\subset M$ be a measurable subset with probability measure $\mu_X$.
Suppose that $f:X\to X$ is a measure-preserving transformation and $h:X\to\R^+$ is
an integrable function such that
$T_{h(x)}x=fx$ for all $x\in X$.

Define the suspension
$X^h=\{(x,u)\in X\times[0,\infty): u\in[0,h(x)]\}/\sim$ where
$(x,h(x))\sim(fx,0)$.   The suspension
semiflow $f_t:X^h\to X^h$ is given by $f_t(x,u)=(x,u+t)$
computed modulo identifications.
We obtain an $f_t$-invariant probability measure on
$X^h$ given by $\mu^h=(\mu_X\times{\rm Lebesgue})/\bar h$ where $\bar h=\int_X h\,d\mu_X$.

Define $\pi:X^h\to M$ by $\pi(x,u)=T_ux$.  This
defines a measurable semiconjugacy from $f_t$ to $T_t$ and $\mu_M=\pi_*\mu^h$ is a $T_t$-invariant probability measure on $M$.

\paragraph{Inducing}
Let $f:X\to X$ be a measurable map defined on a measurable space $X$.
Let $Y\subset X$ be a measurable subset with probability measure $\mu_Y$.
Suppose that $F:Y\to Y$ is a measure-preserving transformation and $\tau:Y\to\Z^+$ is
an integrable function such that
$f^{\tau(y)}y=Fy$ for all $y\in Y$.

Define the tower
$Y^\tau=\{(y,\ell )\in Y\times\Z: 0\le\ell\le\tau(y)-1\}$.
The tower map
$\hat f:Y^\tau\to Y^\tau$ is given by
$\hat f(y,\ell)=\begin{cases} (y,\ell+1) & \ell\le\tau(y)-2
\\ (Fy,0) & \ell=\tau(y)-1 \end{cases}$.
We obtain an $\hat f$-invariant probability measure on
$Y^\tau$ given by $\mu^\tau=(\mu_Y\times{\rm counting})/\bar\tau$ where $\bar\tau=\int_Y \tau\,d\mu_Y$.

Define $\pi:Y^\tau\to X$ by $\pi(y,\ell)=f^\ell y$.  This
defines a measurable semiconjugacy from $\hat f$ to $f$ and $\mu_X=\pi_*\mu^\tau$ is an $f$-invariant probability measure on $X$.

\paragraph{Induced roof functions}
Often these situations are combined, so we have
$T_t:M\to M$, $f:X\to X$ and $F:Y\to Y$ together with
$h:X\to\R^+$ and $\tau:Y\to\Z^+$ such that
$T_h=f$ and $f^\tau=F$.

If $\mu_Y$ is an $F$-invariant probability measure on $Y$ and $\tau$ is integrable, then we can define $\mu_X$ on $X$ as above.  If in addition $h$ is integrable, then we can define $\mu_M$ on $M$ as above.
It is convenient to define the {\em induced roof function}
\begin{align} \label{eq:H}
\SMALL H:Y\to \R^+, \qquad H(y)=\sum_{\ell=0}^{\tau(y)-1}h(f^\ell y).
\end{align}
Then $H$ is integrable (indeed $\int_Y H\,d\mu_Y=\bar h\bar\tau$) and we define the suspension semiflow $F_t:Y^H\to Y^H$ with invariant probability measure $\mu^H=(\mu_Y\times{\rm Lebesgue})/\bH$ where $\bH=\int_Y H\,d\mu_Y$.
Define the
semiconjugacy $\pi:Y^H\to M$ given by $\pi(y,u)=T_uy$.  It is easily checked that $\pi_*\mu^H$ is again $\mu_M$.

\paragraph{More notation}
We use ``big O'' and $\ll$ notation interchangeably.  For example, $a_n=O(b_n)$ or $a_n\ll b_n$ means that
there is a constant $C>0$ such that
$a_n\le Cb_n$ for all $n\ge1$.

\section{Nonuniformly expanding semiflows}
\label{sec:NUE}

Let $(M,d)$ be a bounded metric space with Borel subsets $Y\subset X\subset M$.  Suppose that $T_t:M\to M$, $h:X\to\R^+$, $f=T_h:X\to X$ are as
in Subsection~\ref{sec:prelim}.
Also, let $\tau:Y\to\Z^+$ and $F=f^\tau:Y\to Y$ be as
in Subsection~\ref{sec:prelim}.

Now let $\mu_0$ be a finite Borel measure on $X$ with $\mu_0(Y)>0$.
We suppose that $f$ is a nonsingular transformation for which $\mu_0$ is ergodic, and that $\tau\in L^1(Y)$.
Moreover, we suppose that
there is an at most countable measurable partition $\{Y_j\}$ such that
$\tau$ is constant on partition elements,
and there exist constants $\lambda>1$, $C\ge1$, $\eta\in(0,1]$ such that for all $j\ge1$,
\begin{itemize}
\item[(1)] $F$ restricts to a (measure-theoretic) bijection from $Y_j$  onto $Y$.
\item[(2)] $d(Fy,Fy')\ge \lambda d(y,y')$ for all $y,y'\in Y_j$.
\item[(3)] $\zeta_0=\frac{d\mu_0|_Y}{d\mu_0|_Y\circ F}$
satisfies $|\log \zeta_0(y)-\log \zeta_0(y')|\le Cd(Fy,Fy')^\eta$ for all $y,y'\in Y_j$.
\end{itemize}

There is a unique ergodic $F$-invariant
absolutely continuous probability measure $\mu_Y$ on $Y$.   Moreover, $d\mu_Y/d\mu_0\in L^\infty(Y)$ and
$F:Y\to Y$ is a (full branch) Gibbs-Markov map.
Standard references for material on Gibbs-Markov maps are~\cite[Chapter~4]{Aaronson} and~\cite{AaronsonDenker01}.

Starting from $\mu_Y$ and using that $\tau$ is integrable, we can construct an ergodic $f$-invariant probability measure $\mu_X$ on $X$
as in Section~\ref{sec:prelim}.  We assume that $\int_X h\,d\mu_X$ is integrable, leading to an ergodic $T_t$-invariant probability measure $\mu_M$ on $M$.

For $\theta\in(0,1)$, define the symbolic metric $d_\theta(y,y')=\theta^{s(y,y')}$ where the {\em separation time} $s(y,y')$ is the least integer $n\ge0$ such that $F^ny$ and $F^ny'$ lie in distinct partition elements $Y_j$.
A function $v:Y\to\R$ is {\em $d_\theta$-Lipschitz} if
$|v|_\theta=\sup_{y\neq y'}|v(y)-v(y')|/d_\theta(y,y')$ is finite.
Let $F_\theta(Y)$ be the Banach space of Lipschitz functions with norm $\|v\|_\theta=|v|_\infty+|v|_\theta$.
More generally (and with a slight abuse of notation), we say that a function
$v:Y\to\R$ is {\em piecewise Lipschitz} if
$|1_{Y_j}v|_\theta=\sup_{y,y'\in Y_j,\,y\neq y'}|v(y)-v(y')|/d_\theta(y,y')$
is finite for all $j$.

Let $P:L^1(Y)\to L^1(Y)$ denote the transfer operator corresponding to
$F:Y\to Y$ (so
$\int_Y Pv\,w\,d\mu_Y=\int_Y v\,w\circ F\,d\mu_Y$ for
all $v\in L^1(Y)$, $w\in L^\infty(Y)$).

\begin{prop} \label{prop:GM}
There exists $\theta,\,\gamma\in(0,1)$, $C>0$, such that
\begin{itemize}
\item[(a)]
$\|P^n v\|_\theta\le C\gamma^n\|v\|_\theta$
for all $v\in F_\theta(Y)$ with $\int_Y v\,d\mu_Y=0$ and all $n\ge1$.
\item[(b)] If $g\in L^1(Y)$ such that $|1_{Y_j}v|_\infty \le \inf_{Y_j}g$ and $|1_{Y_j}v|_\theta\le \inf_{Y_j}g$ for all $j\ge1$,  then $\|Pv\|_\theta\le C|g|_1$.
\end{itemize}
\end{prop}

\begin{proof}
Part (a) is proved for instance in~\cite{Aaronson, AaronsonDenker01}.
Part (b) is
 a standard argument (eg.~\cite[Lemma~2.2]{MN05}).
\end{proof}

Define the induced roof function $H:Y\to\R^+$ as in~\eqref{eq:H}.
We assume that
$H$ is piecewise $d_\theta$-Lipschitz for some $\theta\in(0,1)$.
Moreover, we suppose that there is a constant $C\ge1$ such that
\begin{align} \label{eq:inf}
|1_{Y_j}H|_\theta\le C\,{\SMALL\inf_{Y_j}}H \quad\text{for $j\ge1$,}
\end{align}
and
\begin{align} \label{eq:holder}
d(T_ty,T_ty')\le C\theta^{s(y,y')}
\;\text{for $y,y'\in Y_j$, $j\ge1$,
$t\in[0,H(y)]\cap [0,H(y')]$.}
\end{align}

\subsection{Martingale-coboundary decomposition}

Let $v:M\to\R^d$ and define the induced observable
\[
\SMALL V:Y\to\R^d, \qquad V(y)=\int_0^{H(y)}v(T_uy)\,du.
\]

\begin{prop} \label{prop:VY}
Let $\theta_1=\theta^{\eta}$.
There is a constant $C>0$ such that
\[
|1_{Y_j}V|_\infty\le C|v|_\infty{\SMALL\inf}_{Y_j} H, \qquad
|1_{Y_j}V|_{\theta_1}\le C\|v\|_\eta\, {\SMALL\inf}_{Y_j}H
\quad\text{for $j\ge1$,}
\]
for all $v\in C^\eta(M,\R^d)$.
\end{prop}

\begin{proof}
By~\eqref{eq:inf}, $|1_{Y_j}V|_\infty \le |v|_\infty|1_{Y_j}H|_\infty\ll |v|_\infty\inf_{Y_j}H$.
Let $y,y'\in Y_j$ and suppose without loss that $H(y)\ge H(y')$.  By~\eqref{eq:holder},
$|v(T_uy)-v(T_uy')|\le |v|_\eta\, d(T_uy,T_uy')^\eta
\le C |v|_\eta\,  \theta^{\eta s(y,y')}=C |v|_\eta\, d_{\theta_1}(y,y')$ for $u\in[0,H(y')]$.  Hence by~\eqref{eq:inf},
\begin{align*}
|V(y)- & V(y')|
 \le \int_{H(y')}^{H(y)}|v(T_uy)|\,du+\int_0^{H(y')}|v(T_uy)-v(T_uy')|\,du \\
& \le |v|_\infty(H(y)-H(y'))+C|v|_\eta\, H(y') d_{\theta_1}(y,y')
\ll \|v\|_\eta \,{\SMALL\inf_{Y_j}}H d_{\theta_1}(y,y'),
\end{align*}
as required.
\end{proof}

Let $U$ be the Koopman operator $Uv=v\circ F$.
It is standard that $PU=I$ and $UP=\E(\,\cdot\,|F^{-1}\cB)$
where $\cB$ denotes the underlying $\sigma$-algebra on $Y$.

\begin{prop} \label{prop:m}
Suppose that $H\in L^p(Y)$ where $p\ge 1$.
Let $v\in C^\eta_0(M,\R^d)$.
Then there exists
$m\in L^p(Y,\R^d)$, $\chi\in L^\infty(Y,\R^d)$ such that
$V=m+\chi\circ F-\chi$, $m\in\ker P$.

Moreover, there exists a constant $C>0$ such that
$|m|_p\le C\|v\|_\eta$ and
$|\chi|_\infty\le C\|v\|_\eta$ for all
$v\in C^\eta_0(M,\R^d)$.
\end{prop}

\begin{proof}
 By Proposition~\ref{prop:VY}, $V\in L^p(Y,\R^d)$ with
$|V|_p\le |v|_\infty |H|_p$.

By Proposition~\ref{prop:GM}(b) and
Proposition~\ref{prop:VY},
$\|PV\|_{\theta_1}\ll \|v\|_\eta|H|_1\ll \|v\|_\eta$.
By Proposition~\ref{prop:GM}(a),
$\|P^nV\|_{\theta_1}\ll \gamma^n\|PV\|_{\theta_1}$ for $n\ge1$.
 In particular,  $\chi=\sum_{n=1}^\infty P^nV$ converges in $L^\infty(Y)$, and we obtain $V=m+\chi\circ F-\chi$
with $|\chi|_\infty\ll \|v\|_\eta$ and
$|m|_p\le |V|_p+2|\chi|_\infty\ll \|v\|_\eta$.
Finally, $Pm=PV-\chi+P\chi=0$.
\end{proof}

Let $V_n=\sum_{j=0}^{n-1}V\circ F^j$.

\begin{cor} \label{cor:Burk}
Suppose that $H\in L^p(Y)$ where $p\ge2$.
There is a constant $C>0$ such that
$\big|\max_{1\le k\le n}|V_k|\big|_p\le Cn^{1/2}\|v\|_\eta$
for all $v\in C^\eta_0(M,\R^d)$, $n\ge1$.
\end{cor}

\begin{proof}
By Proposition~\ref{prop:m}, $\E(m|F^{-1}\cB)=UPm=0$.
Hence $\{m\circ F^n,\,n\ge0\}$ is a sequence of reversal martingale differences with respect to the sequence of $\sigma$-algebras $\{F^{-n}\cB,\,n\ge0\}$.
Let $m_n=\sum_{j=0}^{n-1}m\circ F^j$.
Since $p\ge2$, it follows from Burkholder's inequality~\cite[Theorem~3.2 combined with (1.4)]{Burkholder73} that
\[
\big|\max_{1\le k\le n}|m_k|\big|_p
\le {\SMALL\frac{p}{p-1}}|m_n|_p\ll \Big|\Big(\sum_{j=1}^n m^2\circ F^j\Big)^{1/2}\Big|_p
\le n^{1/2}|m|_p
\ll n^{1/2}\|v\|_\eta.
\]

Since $V_n=m_n+\chi\circ F^n-\chi$ and
$\chi\in L^\infty$, we obtain
$\big|\max_{1\le k\le n}|V_k|\big|_p
\le \big|\max_{1\le k\le n}|m_k|\big|_p+2|\chi|_\infty
\ll n^{1/2}\|v\|_\eta$.
\end{proof}

\subsection{Central limit theorems}
\label{sec:CLT}

Let $v\in C^\eta(M,\R^d)$ and define
\[
W_n(t)=n^{-1/2}\int_0^{nt}v\circ T_s\,ds,
\qquad
\BBW_n(t)=n^{-1}\int_0^{nt} \int_0^s (v\circ T_r)\otimes(v\circ T_s)\,dr ds.
\]
Also, define $\hat v,\,\psi:Y^H\to\R^d$,
\[
\hat v(y,u)=v(T_uy), \qquad
\SMALL \psi(y,u)=\int_0^u v(T_sy)\,ds.
\]

\begin{thm}   \label{thm:CLT}
Suppose that $H\in L^2(Y)$.  Let $v\in C^\eta_0(M,\R^d)$.
Then there exist $\Sigma$, $D\in\R^{d\times d}$ with $\Sigma^T=\Sigma$ and $D^T=-D$ such that
$(W_n,\BBW_n)\to_w (W,\BBW)$ in $C([0,\infty),\R^d\times\R^{d\times d})$
on $(M,\mu_M)$,
where $W$ is a $d$-dimensional Brownian motion with covariance $\Sigma$ and
$\BBW(t)=\int_0^t W\otimes\circ dW+Dt$.
\end{thm}

\begin{proof}
There are two main steps where
we work with
(i) the induced observable $V:Y\to\R^d$,
(ii) the observable $v:M\to\R^d$.
(This is more direct than the procedure indicated in~\cite[Section~10.3]{KM16}.)

Define
\[
W_n^Y(t)=n^{-1/2}\sum_{0\le j\le [nt]-1}V\circ F^j,
\qquad
\BBW_n^Y(t)=n^{-1}\sum_{0\le i<j\le [nt]-1}(V\circ F^i)\otimes(V\circ F^j).
\]
Recall that the induced map $F:Y\to Y$ is mixing.
Using the decomposition in Proposition~\ref{prop:m} with $p=2$, it follows
from~\cite[Theorem~4.3]{KM16} that
$(W_n^Y,\BBW_n^Y)\to (W^Y,\BBW^Y)$ in $D([0,\infty),\R^d\times\R^{d\times d})$
on $(Y,\mu_Y)$
where $W^Y$ is a $d$-dimensional Brownian motion with some covariance
matrix $\Sigma_Y\in\R^{d\times d}$
and
\[
\BBW^Y(t)=\int_0^t W^Y\otimes \,dW^Y+E_Yt \quad\text{where}\quad
E_Y\in \R^{d\times d}.
\]
This completes step (i).

Next, we claim that
\begin{itemize}
\item[(a)] $V\in L^2(Y,\R^d)$,\hspace{1.5em} $|\psi| |\hat v|\in L^1(Y^H)$.
\item[(b)] $n^{-1/2}\sup_{t\in[0,T]}|\psi\circ f_{nt}|\to_w0$.
\item[(c)] $\lim_{n\to\infty} n^{-1}\big|\max_{1\le k\le n}|\sum_{1\le j\le k} V\circ F^j|\big|_2=0$.
\end{itemize}
To verify the claim, we proceed as in~\cite[Remark~6.2]{KM16}.
First, $H\in L^2(Y)$ so $V\in L^2(Y,\R^d)$ by Proposition~\ref{prop:VY}.  Also
$|\psi(y,u)|\le |v|_\infty H(y)$, so
$\int_{Y^H}|\psi|\,d\mu^H\le |v|_\infty\bH^{-1}\int_YH^2\,d\mu<\infty$.  Hence (a) is satisfied.
Item~(c) is immediate by Corollary~\ref{cor:Burk}.
To prove~(b), we argue as in the proof of~\cite[Proposition~6.6(b)]{KM16}.
Define $g:Y\to\R$, $g(y)=\int_0^{H(y)}|v(T_sy)|\,ds$.
Then $g(y)\le |v|_\infty H(y)$ so $g^2\in L^1(Y)$.
By the ergodic theorem,
$n^{-1/2}g\circ F^n\to0$ a.e.  Define the lap number $N(t):Y^H\to \N$ to be the largest integer $n\ge0$ such that $u+\sum_{j=0}^{n-1}H(F^jy)\le t$.
Then $N(nt)/n\to \bH^{-1}$ a.e.\ and so
$n^{-1/2}g\circ F^{N(nt)}=(N(nt)/n)^{1/2}N(nt)^{-1/2}
g\circ F^{N(nt)}\to0$ a.e.
Also, $|\psi(y,u)|\le g(y)$, so
\[
\SMALL |\psi\circ f_t(y,u)|=\big|\psi\big(F^{N(t)(y,u)}y,u+t-\sum_{j=0}^{N(t)(y,u)-1}H(F^jy)\big)\big|
\le g(F^{N(t)(y,u)}y).
\]
Hence $n^{-1/2}\psi\circ f_{nt}\to0$ a.e.\ and (b) follows.
This completes the verification of the claim.

Since the claim holds, the hypotheses of~\cite[Theorem~6.1]{KM16} are satisfied
(with $V$, $\psi$, $\hat v$ here playing the roles of $\tilde v$, $H$, $v$ in~\cite[Theorem~6.1]{KM16}).
It follows from~\cite[Theorem~6.1]{KM16} that the iterated WIP for $V$ on $Y$ implies an iterated WIP for $\hat v$ on $Y^H$ and hence,
via the semiconjugacy $\pi:Y^H\to M$, an iterated WIP for $v$ on $M$.
We conclude that
$(W_n,\BBW_n)\to_w (W,\BBW)$
in $C([0,\infty),\R^d\times\R^{d\times d})$ on $(M,\mu_M)$,
where
\[
W=\bH^{-1/2}W^Y,\qquad
\SMALL \BBW(t)=\int_0^t W\otimes dW+Et, \qquad
\SMALL E=\bH^{-1}E_Y
+\int_{Y^H} \psi\otimes \hat v\,d\mu^H.
\]
The Stratonovich correction gives $D=E-\frac12\Sigma$ skew-symmetric
as in the proof of~\cite[Corollary~8.1]{KM16}.
\end{proof}

\begin{rmk}
By~\cite[Theorem~6.1]{KM16},
\begin{align*}
\Sigma & =\bH^{-1}\Big(\int_Y V\otimes V\,d\mu_Y+\sum_{n=1}^\infty
\int_Y\big\{V\otimes(V\circ F^n)+(V\circ F^n)\otimes V\big\}\,d\mu_Y\Big),
\\
2D & =\bH^{-1}\sum_{n=1}^\infty\int_Y \{V\otimes(V\circ F^n)-(V\circ F^n)\otimes V\}\,d\mu_Y
+\int_{Y^H}(\psi\otimes \hat v-\hat v\otimes\psi)\,d\mu^H.
\end{align*}
If in addition the semiflow $T_t$ is mixing, then the formulas in
Remark~\ref{rmk:SigmaD} hold.
\end{rmk}

\begin{prop}  \label{prop:deg}
There is a closed subspace $S\subset C_0^\eta(M,\R^d)$ of infinite codimension
such that $\det\Sigma>0$ for all $v\in C_0^\eta(M,\R^d)\setminus S$.
\end{prop}

\begin{proof}
Let $\Sigma_Y$ be as in the proof of Theorem~\ref{thm:CLT},
so $\Sigma=\bH^{-1}\Sigma_Y$.
Suppose that $\det\Sigma=0$, so $\det\Sigma_Y=0$.  Then there is a nonzero vector $c\in\R^d$ such that $c^T\Sigma_Y c=0$.

It is standard that
$\Sigma_Y=\int_Y m\otimes m\,d\mu_Y$
(see for example the fourth line of the proof of~\cite[Theorem~4.3]{KM16}),
so $\int_Y(c\cdot m)^2\,d\mu_Y=0$ and hence $c\cdot m=0$ a.e.
By Proposition~\ref{prop:m}, $c\cdot V=g\circ F-g$ a.e. where $g=c\cdot\chi$.
Moreover, it follows from the proof of Proposition~\ref{prop:m} that
$\chi:Y\to \R^d$ is $d_{\theta_1}$-Lipschitz and hence extends uniquely to
a $d_{\theta_1}$-Lipschitz function on $Y$.  It follows that the equation $c\cdot V=g\circ F-g$ extends to the whole of $Y$.

For $q\ge1$, let $P_q$ denote the set of points $y\in Y$ such that $F^qy$ is well-defined and $F^q y= y$.
Since $F:Y\to Y$ is full-branch and Markov, such
periodic points are dense in $Y$.
Note that $c\cdot\sum_{j=0}^{q-1}V(F^jy)=g(F^qy)-g(y)=0$ for $y\in P_q$.
It follows that $c\,{\cdot} \int_0^{H_q(y)}v(T_ty)\,dt=0$ where
$H_q(y)=\sum_{j=0}^{q-1}H(F^jy)$.
Hence, for each periodic orbit $\{y,Fy,\dots,F^{q-1}y\}\subset P_q$, $q\ge1$, we obtain a distinct nontrivial linear constraint on~$v$.
Now define $S=\{v\in C_0^\eta(M,\R^d):
c\,{\cdot}\int_0^{H_q(y)}v(T_ty)\,dt=0 \;\text{for all $y\in P_q$, $q\ge1$}\}$.
\end{proof}

\subsection{Moment estimates}
\label{sec:moment}

Given $v\in C^\eta(M,\R^d)$, define
\begin{align} \label{eq:vS}
v_t=\int_0^t v\circ T_s\,ds, \qquad
S_t=\int_0^t \int_0^s (v\circ T_r)\otimes(v\circ T_s)\,dr\,ds.
\end{align}

\begin{thm} \label{thm:moment}
\begin{itemize}
\item[(a)] If $H\in L^p(Y)$ for some $p\ge2$ and $\inf H>0$, then there exists $C>0$ such that
\[
|v_t|_{p-1}\le
Ct^{1/2}\|v\|_{\eta}
\quad\text{for all $v\in C^{\eta}_0(M,\R^d)$, $t>0$.}
\]
\item[(b)] If $H\in L^p(Y)$ for some $p\ge4$ and $\inf H>0$, then there exists $C>0$ such that
\[
|S_t|_{(p-1)/2}\le Ct\|v\|_{\eta}^2
\quad\text{for all $v\in C^{\eta}_0(M,\R^d)$, $t>0$.}
\]
\end{itemize}
\end{thm}

\begin{proof}
(a) For $t\in[0,1]$, we have $|v_t|_\infty\le t|v|_\infty\le t^{1/2}|v|_\infty$ so we suppose from now on that $t\ge1$.
Recall that $\hat v=v\circ\pi:Y^H\to\R^d$.
To prove the result, it is equivalent to show that
$|\hat v_t|_{L^{p-1}(Y^H)}\le Ct^{1/2}\|v\|_{\eta}$
where $\hat v_t=\int_0^t\hat v\circ F_s\,ds$.

For $t>0$, recall from the proof of Theorem~\ref{thm:CLT} that the lap number $N(t):Y^H\to\N$ is the largest integer $n\ge0$ such that $u+\sum_{j=0}^{n-1}H(F^jy)\le t$.
Then 
\[
\hat v_t(y,u)=V_{N(t)(y,u)}(y)+\psi\circ F_t(y,u)-\psi(y,u)
\quad\text{for $(y,u)\in Y^{H}$,}
\]
where $\psi(y,u)=\int_0^u v(T_sy)\,ds$.
Also, $|N(t)|_\infty\le at+1$ where $a=(\inf H)^{-1}>0$ by assumption.
Let $V^*_t(y)=\sup_{u\in[0,H(y)]}|V_{N(t)(y,u)}(y)|$.
By Corollary~\ref{cor:Burk}, on $Y$,
\[
|V^*_t|_p\le \big|\max_{k\le at+1}|V_k|\big|_p\ll t^{1/2}\|v\|_\eta.
\]
Write $p^{-1}+q^{-1}=1$.  By H\"older's inequality,
\begin{align*}
\int_{Y^H}|V_{N(t)(y,u)}(y)|^{p-1}\,d\mu^H(y,u)
& \le \bH^{-1}\int_Y H|V^*_t|^{p-1}\,d\mu_Y
\le \bH^{-1}|H|_p|V^*_t|_{q(p-1)}^{p-1}
\\ & = \bH^{-1} |H|_p|V^*_t|_p^{p-1} \ll  t^{(p-1)/2}\|v\|_\eta^{p-1}.
\end{align*}
Also, $|\psi\circ F_t|_{p-1}=|\psi|_{p-1}$, and $|\psi|\le H|v|_\infty$.
Using again that $\int_{Y^H} H^{p-1}\,d\mu^H=\bH^{-1}\int_Y H^p\,d\mu_Y$, we have that $|\psi\circ F_t|_{p-1}\ll |v|_\infty$.  Hence
$|\tilde v_t|_{p-1}\ll t^{1/2}\|v\|_\eta$.
\\[.75ex]
(b)
The steps are analogous to those in part~(a).
It suffices to show that
$|\tilde S_t|_{L^{(p-1)/2}(Y^{H})}\le Ct\|v\|_\eta^2$ for $t\ge1$
where $\tilde S_t=
\int_0^t \int_0^s (\hat v\circ F_r)\otimes(\hat v\circ F_s)\,dr\,ds$.

First, we establish the corresponding estimate on $Y$.
Let
\[
Q_k= \sum_{0\le i<j<k}(V\circ F^i)\otimes(V\circ F^j).
\]
By Proposition~\ref{prop:VY},
$V=m+\chi\circ F-\chi$ where
$m\in L^p(Y)$, $\chi\in L^\infty(Y)$ and $m\in\ker P$.
Since $V_n=m_n+\chi\circ F^n-\chi$, it is easily checked that
\[
\big|\max_{0\le k\le n}|Q_k|\big|=A_1+A_2,
\]
where
\[
A_1=
\max_{0\le k\le n}\big|\sum_{0\le i<j<k}(m\circ F^i)\otimes(m\circ F^j)\big|,
\quad |A_2|_{p/2}\ll n(|m|_{p/2}+|V|_{p/2})|\chi|_\infty\ll n\|v\|_\eta^2.
\]
Next,
\[
A_1=\max_{k\le n}\big|\sum_{i=0}^{k-2}X_{i,k}\big|, \qquad
X_{i,k}=\Big(\sum_{j=1}^{k-i-1}m\otimes(m\circ F^j)\Big)\circ F^i.
\]
Now $P^{i+1}X_{i,k}=P\sum_{j=1}^{k-i-1}m\otimes(m\circ F^j)=
\sum_{j=1}^{k-i-1}(Pm)\otimes(m\circ F^{j-1})=0$.
This means that $X_{i,k}$ is a sequence of reverse martingale differences with respect to the filtration $\{F^{-i}\cB,\,i\ge0\}$.  Since $p\ge4$, it follows from Burkholder's inequality that
\[
|A_1|_{p/2}\ll n^{1/2} \max_{i<k\le n}|X_{i,k}|_{p/2}.
\]
Applying Burkholder's inequality once more,
\[
|X_{i,k}|_{p/2}\le |m|_p \big|\max_{j\le n}|m_j|\big|_p\ll n^{1/2}|m|_p^2
\ll n^{1/2}\|v\|_\eta^2.
\]
Hence, on $Y$,
\[
\big|\max_{0\le k\le n}|Q_k|\big|_{p/2}\ll n\|v\|_\eta^2.
\]

To pass this estimate from $Y$ to $Y^H$ (and hence $M$), we use calculations from~\cite[Proposition~7.5]{KM16}.
Define $\psi(y,u)=\int_0^u v(T_sy)\,ds$ and the lap number $N(t)$
as in part~(a).
Then $\hat v_t=V_{N(t)}+\psi\circ F_t-\psi$ and
\begin{align*}
S_t  & = \int_0^t \hat v_s\otimes(\hat v\circ F_s)\,ds
=\int_0^t V_{N(s)}\otimes(\hat v\circ F_s)\,ds
+\int_0^t (\psi\otimes\hat v)\circ F_s\,ds
-\psi\otimes\hat v_t
\\ & =Q_{N(t)}
-\hat v\otimes\psi
+V_{N(t)}\otimes(\psi\circ F_t)
+\int_0^t (\psi\otimes \hat v)\circ F_s\,ds
-\psi\otimes\hat v_t.
\end{align*}
Hence working on $Y^H$, we have
\begin{align*}
\nonumber
|S_t|_{(p-1)/2}  \le
|Q_{N(t)}|_{(p-1)/2}+ 
|v|_\infty^2 |H|_{(p-1)/2}
& +|N(t)|_\infty|V|_{p-1}|v|_\infty|H|_{p-1}
\\ & 
+2t|v|_\infty^2|H|_{(p-1)/2}
 \ll |Q_{N(t)}|_{(p-1)/2}+t|v|_\infty^2.
\end{align*}
Finally, setting $Q_t^*(y)=\sup_{u\in[0,H(y)]}|Q_{N(t)(y,u)}(y)|$, 
\begin{align*}
\int_{Y^H}|Q_{N(t)}|^{(p-1)/2}d\mu^H
& \le \bH^{-1}\int_Y H|Q_t^*|^{(p-1)/2}d\mu_Y
\le \bH^{-1}|H|_{L^p(Y)}\big|Q_t^*\big|_{L^{p/2}(Y)}^{(p-1)/2}
\\ & \ll \Big|\max_{0\le k\le at+1}|Q_k|\Big|_{L^{p/2}(Y)}^{(p-1)/2}
\ll (t\|v\|_\eta^2)^{(p-1)/2},
\end{align*}
and we obtain $|S_t|_{(p-1)/2}\ll t\|v\|_\eta^2$ as required.
\end{proof}

\begin{cor}
If $p\ge4$, then the formulas for $\Sigma$ and $E$ in Theorem~\ref{thm:WIP} hold.
\end{cor}

\begin{proof}
If $p>2$ then $W_n(1)=n^{-1/2}v_n\to W(1)$ by Theorem~\ref{thm:CLT}
and $|W_n(1)|_q$ is bounded for some $q>2$ by Theorem~\ref{thm:moment}(a).  Hence (cf.~Corollary~\ref{cor:conv})
$\int_M W_n(1)\otimes W_n(1)\,d\mu_M\to \E(W(1)\otimes W(1))=\Sigma$.
Similarly, provided that $p\ge4$, we have $\BBW_n(1)=n^{-1}S_n\to_w \BBW(1)$ and
$|\BBW_n(1)|_q$ bounded for some $q>1$ so
$\int_M \BBW_n(1)\,d\mu_M\to \E(\BBW(1))=E$.
\end{proof}

\section{Rearrangements and improved moment estimates}
\label{sec:rearrange}

The results on central limit theorems in Theorem~\ref{thm:CLT} are optimal; typically such results fail when $H$ is not $L^2$.
However, the moment estimates in Theorem~\ref{thm:moment} are far from optimal.  To rectify this, we introduce the idea of rearrangement.
 The strategy is to replace $\tau$ by a more unbounded return time in such a way that the flow return time $h$ becomes bounded, while keeping $H$ unchanged.

As in Section~\ref{sec:NUE}, $T_t:M\to M$ is a nonuniformly expanding semiflow with cross-section $X=Y^\tau$ and unbounded roof function $h:X\to\R^+$.
We assume that the
 induced roof function $H:Y\to\R^+$ is defined as in~\eqref{eq:H}, with $\inf H>0$, and satisfies~\eqref{eq:inf} and~\eqref{eq:holder}.
We assume also that
\begin{align} \label{eq:holder2}
d(T_tx,T_{t'}x)\le C |t-t'|^\eta
\quad\text{for $t,t'\ge0$, $x\in M$.}
\end{align}

The modified versions $\tilde\tau$ and $\tilde h$ of $\tau$ and $h$ are defined as follows.
First,
\[
\tilde\tau:Y\to\Z^+, \qquad
\tilde\tau|_{Y_j}=\big[\,\|1_{Y_j}H\|_\theta\,\big] +1.
\]
Using $F:Y\to Y$ and $\tilde\tau:Y\to\Z^+$, we construct the tower map
\[
f_\Delta:\Delta\to\Delta, \qquad \Delta=Y^{\tilde\tau},
\]
as in Section~\ref{sec:prelim}.  Now define
\[
\tilde h:\Delta\to\R^+, \qquad
\tilde h(y,\ell)=H(y)/\tilde \tau(y)
\quad\text{for $y\in Y_j$, $0\le\ell<\tilde \tau(y)$.}
\]
The key lemma is the following:
\begin{lemma}    \label{lem:key}
\begin{itemize}
\item[(a)] $H(y)=\sum_{\ell=0}^{\tilde \tau(y)-1}\tilde h(y,\ell)$ for all $y\in Y$.
\item[(b)]  $0<\inf\tilde h\le |\tilde h|_\infty\le1$ and
\[
|\tilde h(y,\ell)-\tilde h(y',\ell)|\le d_\theta(y,y')
\quad \text{for $y,y'\in Y_j$, $j\ge1$, $0\le\ell\le\tilde\tau(y)-1$.}
\]
\end{itemize}
\end{lemma}

\begin{proof}
Part (a) is immediate.
Also, $|\tilde h(y,\ell)-\tilde h(y',\ell)|
=|H(y)-H(y')|/\{[\|1_{Y_j}H\|_\theta]+1\}\le d_\theta(y,y')$
for $y,y'\in Y_j$, $j\ge1$, $0\le\ell\le\tilde\tau(y)-1$,
and similarly $|\tilde h|_\infty\le1$.

It remains to verify that $\tilde h$ is bounded below.
Let $y\in Y_j$, $0\le \ell< \tilde\tau(y)$.
Using~\eqref{eq:inf} and $\inf H>0$,
\begin{align*}
\SMALL \big[\,\|H1_{Y_j}\|_\theta\,\big] +1 & \le |1_{Y_j}H|_\infty+|1_{Y_j}H|_\theta+1
\\[.75ex] & \SMALL\le (2C+1)\inf_{Y_j}H+(\inf H)^{-1}\inf_{Y_j}H\le C'\inf_{Y_j}H,
\end{align*}
where $C'=2C+1+(\inf H)^{-1}$.
Hence $\tilde h(y,\ell)=H(y)/(\big[\,\|H1_{Y_j}\|_\theta\,\big] +1)\ge  1/C'$ as required.
\end{proof}

Since $\tilde h$ is bounded above and below, it is immediate from
Lemma~\ref{lem:key}(a) that
$\mu_Y(H>n)\approx \mu_Y(\tilde\tau>n)$.
As usual, we assume that $H$ (and $\tilde\tau$) is integrable.

As in Section~\ref{sec:prelim}, we
define the suspension semiflow $f_t:\Delta^{\tilde h}\to\Delta^{\tilde h}$
and ergodic invariant probability measures
$\mu_\Delta=\mu_Y^{\tilde\tau}$ on $\Delta$ and $\mu^{\tilde h}$ on $\Delta^{\tilde h}$.

Let $\tilde h_\ell(y)=\sum_{k=0}^{\ell-1}\tilde h(y,k)$ for $\ell\ge1$ and
define
\begin{align} \label{eq:pi}
\pi_\Delta:\Delta^{\tilde h}\to M, \qquad \pi_\Delta(y,\ell,u)=T_{u+\tilde h_\ell(y)}y.
\end{align}

By Lemma~\ref{lem:key}(a), we obtain
\begin{prop}  \label{prop:pi}
$\pi_\Delta$ is a measurable semiconjugacy from $(f_t,\Delta^{\tilde h})$ to $(T_t,M)$.
\end{prop}

\begin{proof}
It suffices to show that $\pi_\Delta$ respects the identifications on $\Delta$ and $\Delta^{\tilde h}$.

Recall that
\[
\SMALL H(y)=h_{\tau(y)}(y)=\sum_{\ell=0}^{\tau(y)-1}h(f^\ell y), \qquad  Fy=f^{\tau(y)}y,\qquad f(y)=T_{h(y)}y,
\]
for $y\in Y$.  It follows that
\[
F=f^\tau=(T_h)^\tau=T_H \quad\mbox{on $Y$}.
\]

First, we verify that $\pi_\Delta(y,\tilde\tau(y),u)=\pi_\Delta(Fy,0,u)$.  By
Lemma~\ref{lem:key}(a),
\[
\pi_\Delta(y,\tilde\tau(y),u) =T_{u+\tilde h_{\tilde\tau(y)}(y)}y
=T_uT_{H(y)}y=T_uFy=\pi_\Delta(Fy,0,u).
\]

Second, we verify that $\pi_\Delta(p,\tilde h(p))=\pi_\Delta(f_\Delta p,0)$.
Write $p=(y,\ell)\in\Delta$.  If $\ell\le\tau(y)-2$, then
\begin{align*}
\pi_\Delta(p,\tilde h(p))
& =\pi_\Delta(y,\ell,\tilde h(y,\ell))=T_{\tilde h(y,\ell)+\tilde h_\ell(y)}y \\ & = T_{\tilde h_{\ell+1}(y)}y
=\pi_\Delta(y,\ell+1,0)=\pi_\Delta(f_\Delta(y,\ell),0).
\end{align*}
Also, if $\ell=\tau(y)-1$, then $\tilde h_{\ell+1}(y)=H(y)$ and $f_\Delta(y,\ell)=(Fy,0)$, so
\begin{align*}
\pi_\Delta(p,\tilde h(p))=T_{H(y)}y=Fy=\pi_\Delta(Fy,0,0)=\pi_\Delta(f_\Delta(y,\tau(y)-1),0)=\pi_\Delta(f_\Delta p,0),
\end{align*}
as required.
\end{proof}

\begin{rmk}  We refer to the procedure in this section as ``rearranging the tower'' because the obvious suspension flow to consider is $(Y^\tau)^h$ as in Subsection~\ref{sec:prelim}.
As evidenced by Theorem~\ref{thm:momentDelta} below, the rearranged suspension $\Delta^{\tilde h}=(Y^{\tilde\tau})^{\tilde h}$ has significant advantages.

This procedure is analogous to introducing transparent walls,
a technique well-known in the theory of billiards (see eg.~\cite{KramliSimanyiSzasz90}). When crossing a transparent wall,
the velocity of the billiard particle is not modified, nonetheless these events can be regarded as (fake) collisions,
which makes the length of the free flight bounded. However, in the present paper, we implement this idea at the more
abstract level of the Young tower, which has certain advantages.
In particular, there is no need for the fake collision set to be identified with a subset of the ambient phase-space.
\end{rmk}

\subsection{Martingale-coboundary decomposition on $\Delta$}

Let $v\in C^\eta(M,\R^d)$ and define
$\tilde v=v\circ\pi_\Delta:\Delta^{\tilde h}\to\R^d$.
Instead of inducing on $Y$, we induce on $\Delta$,
defining the induced observable
\[
\SMALL \tV:\Delta\to\R^d, \qquad
\tV(y,\ell)=\int_0^{\tilde h(y,\ell)}\tilde v(y,\ell,u)\,du.
\]
Clearly, $|\tilde v|_\infty=|v|_\infty$ and
$|\tV|_\infty\le |\tilde h|_\infty|\tilde v|_\infty\le |v|_\infty$.

\begin{prop} \label{prop:tV}
Let $\theta_1=\theta^{\eta^2}$.
There is a constant $C>0$ such that
\[
|\tV(y,\ell)- \tV(y',\ell)| \le C \|v\|_{\eta}\tilde \tau(y)^\eta d_{\theta_1}(y,y') \quad \text{for $y,y'\in Y_j$, $j\ge1$, $0\le\ell\le \tilde\tau(y)-1$,}
\]
for all $v\in C^\eta(M,\R^d)$.
\end{prop}

\begin{proof}
Let $(y,\ell,u),\,(y',\ell,u)\in\Delta^{\tilde h}$ with $y,y'\in Y_j$, $j\ge1$.
Then
\begin{align*}
|\tilde v(y,\ell,u)-\tilde v(y',\ell, & u)|  =
|v(T_{\tilde h_\ell(y)+u}y)- v(T_{\tilde h_\ell(y')+u}y')|
\le |v|_\eta d(T_{\tilde h_\ell(y)+u}y,T_{\tilde h_\ell(y')+u}y')^\eta.
\end{align*}
By Lemma~\ref{lem:key}(b),
$|\tilde h_\ell(y)-\tilde h_\ell(y')|
\le \sum_{k=0}^{\tilde \tau(y)-1} |\tilde h(y,k)-\tilde h(y',k)|\le
\tilde \tau(y) d_\theta(y,y')$,
so
\[
d(T_{\tilde h_\ell(y)+u}y,T_{\tilde h_\ell(y')+u}y)\le
C|\tilde h_\ell(y)-\tilde h_\ell(y')|^\eta
\le C\tilde \tau(y)^\eta d_\theta(y,y')^\eta,
\]
by~\eqref{eq:holder2}.
Also, by~\eqref{eq:holder},
$d(T_{\tilde h_\ell(y')+u}y,T_{\tilde h_\ell(y')+u}y') \le C  \theta^{s(y,y')}$.
Hence,
\[
|\tilde v(y,\ell,u)-\tilde v(y',\ell,u)|\le 2C|v|_{\eta} \tilde \tau(y)^\eta d_{\theta_1}(y,y').
\]

Next,
\begin{align*}
|\tV & (y,\ell)  - \tV(y',\ell)|  \le \int_{\tilde h(y',\ell)}^{\tilde h(y,\ell)}|\tilde v(y,\ell,u)|\,du
+\int_0^{\tilde h(y',\ell)}|\tilde v(y,\ell,u)-\tilde v(y',\ell,u)|\,du \\
& \le |\tilde h(y,\ell)-\tilde h(y',\ell)||v|_\infty+|\tilde h|_\infty
2C|v|_{\eta} \tilde\tau(y)^\eta d_{\theta_1}(y,y')
\ll \|v\|_{\eta} \tilde\tau(y)^\eta d_{\theta_1}(y,y'),
\end{align*}
by Lemma~\ref{lem:key}(b).
\end{proof}

Let $L:L^1(\Delta)\to L^1(\Delta)$ denote the transfer operator corresponding to
$f_\Delta:\Delta\to\Delta$ (so
$\int_\Delta Lv\,w\,d\mu_\Delta=\int_\Delta v\,w\circ f_\Delta\,d\mu_\Delta$ for
all $v\in L^1(\Delta)$, $w\in L^\infty(\Delta)$).

\begin{prop} \label{prop:mDelta}
Assume that $\mu_Y(H>n)=O(n^{-\beta})$ for some $\beta>1$.
Let $v\in C^\eta_0(M,\R^d)$.
  For any $p<\beta-1$, there exists
$\tm,\tchi\in L^p(\Delta,\R^d)$ such that
$\tV=\tm+\tchi\circ f_\Delta-\tchi$, $\tm\in\ker L$.

Moreover, there exists a constant $C>0$ such that
$|\tm|_p\le C\|v\|_\eta$ and $|\tchi|_p\le C\|v\|_\eta$ for all $v\in C_0^\eta(M,\R^d)$.
\end{prop}

\begin{proof}
By Lemma~\ref{lem:key}, $\mu_Y(\tilde\tau>n)=O(n^{-\beta})$.
Note that
\begin{align} \label{eq:taunote}
\mu_\Delta\{(y,\ell)\in\Delta:\tilde\tau(y)>n\}
=({\SMALL\int}_Y\tilde\tau\,d\mu_Y)^{-1}{\SMALL\int}_Y\tilde\tau 1_{\{\tilde\tau>n\}}\,d\mu_Y=O(n^{-(\beta-1)}).
\end{align}

Without loss, $d=1$.
Shrinking $\eta\in(0,1]$ if necessary, we can suppose that $p<\beta-1-\eta$.

Fix $n\ge1$ and let
$\tV_n(y,\ell)=\begin{cases} \tV(y,\ell) & \tilde \tau(y)\le n
\\ 0 & \tilde \tau(y)\ge n+1 \end{cases}$.
Then $|\tV_n|_\infty\le |\tV|_\infty\le |v|_\infty$.
Also, for $y,y'\in Y_j$, $j\ge1$, it follows from Proposition~\ref{prop:tV} that
\[
|\tV_n(y,\ell)-\tV_n(y',\ell)|\le 1_{\{\tau(y)\le n\}}|\tV(y,\ell)-\tV(y,\ell')|
\le C\|v\|_{\eta}\, n^\eta \,d_{\theta_1}(y,y').
\]

This means that $\|\tV_n\|\le C\|v\|_{\eta}\, n^\eta$
in the function space in~\cite{Young99}.
Since $\mu_Y(\tilde\tau>n)=O(n^{-\beta})$, it follows from~\cite[Theorem~3]{Young99} that
\[
\SMALL \big|\int_\Delta \tV_n\,w\circ f_\Delta^n \,d\mu_{\Delta}-
\int_\Delta \tV_n \,d\mu_{\Delta}
\int_\Delta w \,d\mu_{\Delta}\big|\ll
\|\tV_n\| |w|_\infty \,n^{-(\beta-1)}
\le C \|v\|_{\eta} |w|_\infty n^{-(\beta-1-\eta)},
\]
for all $w\in L^\infty(\Delta)$, $n\ge1$.
(The dependence on $\|\tV_n\|$ is not stated explicitly in \cite[Theorem~3]{Young99} but follows by a standard argument using the uniform boundedness principle.  Alternatively, see~\cite{KKMapp} for a direct argument.)
Also
\[
\SMALL \big|\int_\Delta (\tV-\tV_n)\,w\circ f_\Delta^n \,d\mu_{\Delta}\big|
\le |\tV-\tV_n|_\infty|w|_\infty\,\mu_{\Delta}\{(y,\ell):\tilde\tau(y)>n\}
\ll |v|_\infty|w|_\infty n^{-(\beta-1)},
\]
by~\eqref{eq:taunote},
and similarly $|\int_\Delta \tV_n\,d\mu_{\Delta}|=|\int_\Delta (\tV-\tV_n) \,d\mu_{\Delta}|\ll |v|_\infty n^{-(\beta-1)}$
(where we have used that $\int_\Delta \tV\,d\mu_{\Delta}=0$).  Hence
\[
\SMALL \big|\int_\Delta L^n\tV\,w \,d\mu_{\Delta} \big|=
\big|\int_\Delta \tV\,w\circ f_\Delta^n \,d\mu_{\Delta} \big|\ll
\|v\|_{\eta} |w|_\infty \,n^{-(\beta-1-\eta)}.
\]
Taking $w=\sgn L^n\tV$, we obtain $|L^n\tV|_1\ll
\|v\|_{\eta} \,n^{-(\beta-1-\eta)}$.
Hence
\[
|L^n\tV|_p^p\le
|L^n\tV|_\infty^{p-1}|L^n\tV|_1\le |\tV|_\infty^{p-1}|L^n\tV|_1\ll |v|_\infty^{p-1}\|v\|_{\eta}\,n^{-(\beta-1-\eta)},
\]
so $\tchi=\sum_{n=1}^\infty L^n\tV$ converges in $L^p$.
Finally, $L\tm=L\tV-\tchi+L\tchi=0$.
\end{proof}

\subsection{Improved moment estimates}
\label{sec:momentDelta}

Given $v\in C^\eta(M,\R^d)$, define
$v_t$ and $S_t$ as in~\eqref{eq:vS}.
A crucial difference from Section~\ref{sec:moment} is that $\tV$ is $L^\infty$ even though, as before, $\tm$ is $L^p$.
This enables the use of Rio's inequality~\cite{MerlevedePeligradUtev06,Rio00}
following~\cite{MN08}.
We continue to assume that $\mu_Y(H>n)=O(n^{-\beta})$ and that $\inf H>0$.

\begin{thm} \label{thm:momentDelta}
(a) Let $2<p<\beta$.  There is a constant $C>0$ such that
\[
|v_t|_{2(p-1)}\le
Ct^{1/2}\|v\|_{\eta}
\qquad\text{for all $v\in C^{\eta}_0(M,\R^d)$, $t>0$.}
\]
\noindent(b) Let $4<p<\beta$.  There is a constant $C>0$ such that
\[
|S_t|_{2(p-1)/3}\le Ct\|v\|_{\eta}^2
\qquad\text{for all $v\in C^{\eta}_0(M,\R^d)$, $t>0$.}
\]
\end{thm}

\begin{proof}
As in the proof of Theorem~\ref{thm:moment}, we may suppose that $t\ge1$.
Recall that $\tilde v=v\circ\pi_\Delta:\Delta^{\tilde h}\to\R^d$.
To prove the result, it is equivalent to show that
$|\tilde v_t|_{L^{2(p-1)}(\Delta^{\tilde h})}\ll t^{1/2}\|v\|_\eta$
and $|\tilde S_t|_{L^{2(p-1)/3}(\Delta^{\tilde h})}\ll t\|v\|_\eta^2$,
where $\tilde v_t=\int_0^t\tilde v\circ f_s\,ds$ and
$\tilde S_t=
\int_0^t \int_0^s (\tilde v\circ f_r)\otimes(\tilde v\circ f_s)\,dr\,ds$.
\\[.75ex]
(a) We follow the arguments in~\cite[Theorem~3.1]{MN08} and~\cite[Lemma~4.1]{MTorok12}.
Recall that $\tV\in L^\infty(\Delta,\R^d)$.
By Proposition~\ref{prop:mDelta},
$\tV=\tm+\tchi\circ f_\Delta-\tchi$ where $\tm,\tchi$ in $L^{p-1}(\Delta,\R^d)$ and $\tm\in\ker L$.
(Unfortunately, $p-1$ here corresponds to $p$ in~\cite{MN08} and Proposition~\ref{prop:mDelta}, but our notation facilitates comparison of Theorems~\ref{thm:moment} and~\ref{thm:momentDelta}.)
This is the same as the decomposition $\phi=\psi+\chi\circ T-\chi$ in the proof of~\cite[Theorem~3.1]{MN08}.  Hence it follows just as in the proof of~\cite[Theorem~3.1]{MN08} that
$|\tV_n|_{L^{2(p-1)}(\Delta)}\ll n^{1/2}\|v\|_\eta$ where $\tV_n=\sum_{j=0}^{n-1}\tV\circ f_\Delta^j$.
By~\cite[Corollary~B1]{Serfling70} (with $\nu=2(p-1)$ and $\delta=1$), we obtain
$\big|\max_{0\le k\le n}|\tV_k|\big|_{L^{2(p-1)}(\Delta)}\ll n^{1/2}\|v\|_\eta$.

For $t>0$, define the lap number $N(t):\Delta\to\N$ to be the largest integer $n\ge0$ such that $\tilde h_n\le t$.
By Lemma~\ref{lem:key}(b),
\[
\SMALL |\tilde v_t(x,u)-\tV_{N(t)(x)}(x)|\le 2|\tilde v|_\infty|\tilde h|_\infty
\le 2|v|_\infty
\quad\text{for $(x,u)\in\Delta^{\tilde h}$.}
\]
Also $|N(t)|_\infty\le at+1$ where $a=\inf \tilde h^{-1}>0$ by Lemma~\ref{lem:key}(b).
Hence
\begin{align*}
|\tilde v_t|_{L^{2(p-1)}(\Delta^{\tilde h})} & \le
2|v|_\infty + \Big(\int_{\Delta^{\tilde h}}|\tV_{N(t)(x)}|^{2(p-1)}\,d\mu^{\tilde h}(x,u)\Big)^{1/(2(p-1))}
\\ & \le 2|v|_\infty + ({\SMALL\int_\Delta}\tilde h\,d\mu_\Delta)^{-1/(2(p-1))} |\tilde h|_\infty^{1/(2(p-1))} |\tV_{N(t)}|_{L^{2(p-1)}(\Delta)}
\\ & \ll |v|_\infty + \Big|\max_{k\le at+1}|\tV_k|\Big|_{L^{2(p-1)}(\Delta)}
\ll |v|_\infty+ (at+1)^{1/2}\|v\|_\eta\ll t^{1/2}\|v\|_\eta.
\end{align*}
This completes the proof that $|v_t|_{2(p-1)}\ll t^{1/2}\|v\|_\eta$.
\\[.75ex]
(b)
Again, we have
$\tV=\tm+\tchi\circ f_\Delta-\tchi$ where
$\tV\in L^\infty(\Delta,\R^d)$, $\tm,\,\tchi$ in $L^{p-1}(\Delta,\R^d)$ and $\tm\in\ker L$.
This means that the hypotheses
 of~\cite[Proposition~7.1]{KM16} are satisfied and
\[
\Big|\max_{0\le k\le n}\big|\sum_{i=0}^{k-1}\sum_{j=0}^{i-1}(\tV\circ f_\Delta^i)\otimes(\tV\circ f_\Delta^j)\big|\Big|_{2(p-1)/3}\ll n\|v\|_\eta^2.
\]
Using that $\tilde h$ is bounded above and below, this moment estimate for discrete time again passes over to continuous time, see~\cite[Proposition~7.5]{KM16}.
\end{proof}

\begin{rmk}  The moment estimate in Theorem~\ref{thm:momentDelta}(a)
is essentially optimal, see~\cite{M09b,MN08}.
The estimate in Theorem~\ref{thm:momentDelta}(b) is the same as the one obtained in~\cite{KM16} and already significantly improves the one in Theorem~\ref{thm:moment}(b), but we expect that the optimal estimate is
$|S_t|_{p-1}\le Ct\|v\|_\eta^2$.  (Such an estimate is obtained for certain nonuniformly expanding/hyperbolic systems in~\cite{KKMprep}.)
\end{rmk}

\section{Nonuniformly hyperbolic exponentially contracting flows}
\label{sec:exp}

In this section, we prove our main results  for a class of nonuniformly hyperbolic flows.  As shown in Section~\ref{sec:Lorenz}, this class of flows includes Lorenz attractors and singular hyperbolic attractors.

\subsection{Standing assumptions for flows}
\label{sec:stand}

Let $(M,d)$ be a bounded metric space with Borel subsets $Y\subset X\subset M$.  Suppose that $T_t:M\to M$, $h:X\to\R^+$, $f=T_h:X\to X$ are as
in Subsection~\ref{sec:prelim}.
Also, let $\tau:Y\to\Z^+$ and $F=f^\tau:Y\to Y$ be as
in Subsection~\ref{sec:prelim}.

We suppose that $\mu_Y$ is a Borel ergodic $F$-invariant probability measure on $Y$ and that $\tau$ is integrable on $(Y,\mu_Y)$.
Then we obtain
an ergodic $f$-invariant probability measure $\mu_X$ on $X$.
We assume that $h$ is integrable on $(X,\mu_X)$ and hence obtain
an ergodic $T_t$-invariant probability measure $\mu_M$ on $M$.

Let $\cW^s$ be a measurable partition of $Y$ consisting of {\em stable leaves}.
For each $y\in Y$, let $W^s_y$ denote the stable leaf containing $y$.  We require that $F(W^s_y)\subset W^s_{Fy}$ for all $y\in Y$ and that $\tau$ is constant on stable leaves.

Let $\bY$ denote the space obtained from $Y$ after quotienting by $\cW^s$ with natural projection $\bpi:Y\to\bY$.  We assume that the quotient map $\bF:\bY\to\bY$ is Gibbs-Markov, with partition $\{\bY\!_j\}$ and ergodic invariant probability measure $\bmu=\bpi_*\mu_Y$.  Let $s(y,y')$ denote the separation time on $\bY$.

Let $Y_j=\bpi^{-1}\bY\!_j$; these form a partition of $Y$ and each $Y_j$ is a union of stable leaves.  The separation time extends to $Y$, setting $s(y,y')=s(\bpi y,\bpi y')$ for $y,y'\in Y$.

\subsection{Exponential contraction along stable leaves}
\label{subsec:exp}

Assume the set up in Subsection~\ref{sec:stand}.
Define the induced roof function $H:Y\to\R^+$ as in~\eqref{eq:H}.  We suppose throughout that $H$ is integrable and that $\inf H>0$.

We require in addition that there is a measurable subset $\hY\subset Y$ such that for every $y\in Y$ there is a unique $\hat y\in\hY\cap W^s_y$.
Assume that there are constants $C\ge1$, $\gamma_0\in(0,1)$,
$\eta\in(0,1]$,
such that
\begin{alignat}{2}
\label{eq:inf_exp}
|H(y)-H(y')| & \le C({\SMALL\inf_{Y_j}}H)\gamma_0^{s(y,y')}
&& \quad\text{for $y,y'\in \hY\cap Y_j$, $j\ge1$,}
\\
\label{eq:holder_exp}
 d(T_ty,T_{t'}y) & \le C |t-t'|^\eta
&& \quad\text{for $y\in \hY$, $t,t'\ge0$,}
\\
\label{eq:Wsexp}
d(T_ty,T_ty') & \le C\gamma_0^t && \quad\text{for $y\in Y$, $y'\in W^s_y$, $t\ge0$}, \\ \label{eq:hY}
d(T_ty,T_ty') & \le C\gamma_0^{s(y,y')-J(t)(y)} && \quad\text{for $y,y'\in \hY\cap Y_j$, $j\ge1$, $t\ge0$},
\end{alignat}
where $J(t)(y)$ is the largest integer $n\ge0$ such that $H_n(y)\le t$
(as usual $H_n=\sum_{j=0}^{n-1}H\circ F^j$).
These conditions are standard
except for the relatively restrictive condition~\eqref{eq:Wsexp} which represents exponential contraction of stable leaves under the full flow (and not just under the induced map $F$).

\subsection{Martingale-coboundary decomposition}

Given $v\in C^\eta(M,\R^d)$, define the induced observable $V(y)=\int_0^{H(y)}v(T_uy)\,du$ on $Y$.
Also, define
$\chi_1,\,\hV:Y\to\R^d$,
\[
\SMALL
\chi_1(y)=\int_0^\infty (v(T_t\hat y)-v(T_ty))\,dt,
\qquad
\hV=V+\chi_1-\chi_1\circ F.
\]
Finally, define
$g:\bY\to\R$ by $g|_{\bY\!_j} = {\SMALL\inf}_{Y_j}H$, $j\ge1$.

\begin{prop} \label{prop:sinai}
Suppose that $H\in L^p(Y)$ for some $p\ge1$.
There exist constants $C>0$, $\theta\in(0,1)$ such that for all $v\in C^\eta(M,\R^d)$,
\begin{itemize}
\item[(a)] $\chi_1\in L^\infty(Y,\R^d)$ and $|\chi_1|_\infty\le C|v|_\eta$,
\item[(b)] $\hV=\bV\circ\bpi$ for some $\bV\in L^p(\bY,\R^d)$.
Moreover, $|\bV|\le C\|v\|_\eta\, g$ and
$|\bV(y)-\bV(y')|\le C\|v\|_\eta\, g(y)d_\theta(y,y')$ for
$y,y'\in\bY\!_j$, $j\ge1$.
\end{itemize}
\end{prop}

\begin{proof}
(a) Let $\gamma_1=\gamma_0^\eta$.  By~\eqref{eq:Wsexp},
\[
\SMALL |\chi_1(y)|\le |v|_\eta \int_0^\infty d(T_t\hat y,T_ty)^\eta\,dt
\le C|v|_\eta \int_0^\infty \gamma_1^t\,dt\ll |v|_\eta,
\]
so $|\chi_1|_\infty\ll |v|_\eta$.
\\
(b)  A computation shows that
\[
\SMALL \hV=V(\hat y)+\int_0^\infty (v(T_tF\hat y)-v(T_t\widehat{Fy}))\,dt.
\]
In particular, $\hV$ is constant along stable leaves and has the desired factorisation $\hV=\bV\circ\bpi$.
It is immediate from the definition of $\hV$ that
$|\hV|\le |V|+2|\chi_1|_\infty\le |v|_\infty H+2|\chi_1|_\infty\ll \|v\|_\eta\,H$.
Hence $|\bV|\ll \|v\|_\eta\, g$ and $\bV\in L^p(\bY,\R^d)$.

Next, let $y,y'\in Y$ and set $M=s(y,y')/q$ where $q>1$ is to be specified.  Then
\[
|\hV(y)-\hV(y')|\le |V(\hat y)-V(\hat y')|+A_M(y)+A_M(y')+B_M(F\hat y,F\hat y')+B_M(\widehat{Fy},\widehat{Fy'}),
\]
where
\begin{align*}
A_M(z) & = \SMALL\int_M^\infty |v(T_tF\hat z)-v(T_t\widehat{Fz})|\,dt, \qquad
B_M(z,z')  = \SMALL\int_0^M |v(T_tz)-v(T_tz')|\,dt .
\end{align*}
The same calculation as in (a) shows that
$A_M(z)\ll |v|_\eta \gamma_1^M$ for $z=y,y'$.

Next, let $y,y'\in\hY\cap Y_j$.  By~\eqref{eq:inf_exp},~\eqref{eq:holder_exp} and~\eqref{eq:hY},
\begin{align*}
d(T_tF\hat y,T_t & F\hat y') =
d(T_tFy,T_tFy') =
d(T_{t+H(y)}y,T_{t+H(y')}y') \\
& \le  d(T_{t+H(y)}y,T_{t+H(y)}y')
+ d(T_{t+H(y)}y',T_{t+H(y')}y') \\
& \le C\gamma_0^{s(y,y')-J(t+H(y))(y)} +|H(y)-H(y')|^\eta
 \ll ({\SMALL\inf}_{Y_j}H)\gamma_1^{s(y,y')-J(t+H(y))(y)}.
\end{align*}
Hence, setting $\gamma_2=\gamma_1^{\eta}$,
\begin{align*}
B_M(F\hat y,F\hat y') & \le
|v|_\eta \int_0^M d(T_tF\hat y,T_tF\hat y')^\eta\,dt
\ll |v|_\eta\, M({\SMALL \inf}_{Y_j}H) \gamma_2^{s(y,y')-J(M+H(y))(y)}.
\end{align*}
Now $J(t)\le t/\inf H$ and $J(t+H(y))(y)=1+J(t)$.
Hence
\begin{align*}
B_M(F\hat y,F\hat y')
& \ll |v|_\eta\, M({\SMALL \inf}_{Y_j}H) \gamma_2^{s(y,y')-(1+M/\inf H)}
\\ & \ll |v|_\eta ({\SMALL \inf}_{Y_j}H)s(y,y')\gamma_2^{(1-1/(q\inf H))s(y,y')}.
\end{align*}
A much simpler calculation shows that
\[
B_M(\widehat{Fy},\widehat{Fy'}) \le C |v|_\eta M \gamma_1^{s(Fy,Fy')-M/\inf H}
\ll |v|_\eta s(y,y')\gamma_1^{(1-1/(q\inf H))s(y,y')}.
\]
Finally, proceeding as in the proof of Proposition~\ref{prop:VY} and using~\eqref{eq:inf_exp} and~\eqref{eq:hY},
\begin{align*}
|V(\hat y)-V(\hat y')| & \le |v|_\infty|H(\hat y)-H(\hat y')|
+C|v|_\eta \min\{H(\hat y),H(\hat y')\}\gamma_1^{s(y,y')-1}
\\ & \ll \|v\|_\eta ({\SMALL\inf}_{Y_j}H)\gamma_1^{s(y,y')}.
\end{align*}

Altogether, we have shown that
$|\hV(y)-\hV(y')|\ll \|v\|_\eta ({\SMALL\inf}_{Y_j}H) (\gamma_1^{s(y,y')/q}+s(y,y')\gamma_2^{(1-1/(q\inf H))s(y,y')})$.
Taking $q=\eta^{-1}+( \inf H)^{-1}$ and $\theta\in(\gamma_1^{1/q},1)$,
we obtain the desired result.
\end{proof}

\begin{cor} \label{cor:m_exp}
Let $v\in C^\eta_0(M,\R^d)$ and suppose that $H\in L^p(Y)$ for some $p\ge1$.
Then there exists
$\bm\in L^p(\bY,\R^d)$, $\chi_1\in L^\infty(Y,\R^d)$ such that
$V=\bm\circ\bpi+\chi_1\circ F-\chi_1$ and $\bm\in\ker P$.

Moreover, there exists $C>0$ such that
$|\bm|_p\le C\|v\|_\eta$ and
$|\chi_1|_\infty\le C\|v\|_\eta$ for all $v\in C^\eta_0(M,\R^d)$.
\end{cor}

\begin{proof}
By Proposition~\ref{prop:sinai}, $V=\hV+\chi_1\circ F-\chi_1$
where $\chi_1\in L^\infty(Y,\R^d)$ and
$\hV=\bV\circ\bpi$ satisfies the estimates in
Proposition~\ref{prop:sinai}(b).
By Proposition~\ref{prop:GM}(b),
$\|P\bV\|_\theta\ll |g|_1\le |H|_1$.
Using Proposition~\ref{prop:GM}(a), we can proceed as in Proposition~\ref{prop:m}, defining
$\chi_2=\sum_{n=1}^\infty P^n\bV\in L^\infty(\bY,\R^d)$.
Then $\bV=\bm+\chi_2\circ \bF-\chi_2$ where $\bm\in\ker P$.
Now define $\chi=\chi_1+\chi_2\circ\bpi$.
\end{proof}

\begin{cor} \label{cor:Burk_exp}
Suppose that $H\in L^p(Y)$ for some $p\ge2$.  Then there is a constant $C>0$ such that
$\big|\max_{1\le k\le n}|V_k|\big|_p\le Cn^{1/2}\|v\|_\eta$
for all $v\in C^\eta_0(M,\R^d)$, $n\ge1$.
\end{cor}

\begin{proof}  As in the proof of Corollary~\ref{cor:Burk},
$\big|\max_{1\le k\le n}|(\bar m\circ\bpi)_k|\big|_{L^p(Y)}=
\big|\max_{1\le k\le n}|\bar m_k|\big|_{L^p(\bY)}\ll n^{1/2}$.
The result follows.
\end{proof}

\subsection{Central limit theorems and moment estimates}

We continue to assume the set up in Subsections~\ref{sec:stand} and~\ref{subsec:exp}, including that $\inf H>0$.

\begin{thm}   \label{thm:CLT_exp}
Suppose that $H\in L^2(Y)$.
Let $v\in C^\eta_0(M,\R^d)$.
Then the statement and conclusion of Theorem~\ref{thm:CLT} holds.
That is, $(W_n,\BBW_n)\to_w (W,\BBW)$ in $C([0,\infty),\R^d\times\R^{d\times d})$
on $(M,\mu_M)$, where $W_n$, $\BBW_n$, $W$ and $\BBW$ are defined as in
Theorem~\ref{thm:CLT}.
\end{thm}

\begin{proof}
Define $W_n^Y$, $\BBW_n^Y$ on $(Y,\mu_Y)$ as in the proof of Theorem~\ref{thm:CLT}.
In place of Proposition~\ref{prop:m} and~\cite[Theorem~4.3]{KM16}, we use
Corollary~\ref{cor:m_exp} and~\cite[Theorem~5.2]{KM16} to deduce that
$(W_n^Y,\BBW_n^Y)\to(W^Y,\BBW^Y)$ in $D([0,\infty),\R^d\times\R^{d\times d})$
where $W^Y$ and $\BBW^Y$ are as in step~(i) of the proof of Theorem~\ref{thm:CLT}.

Claims (a)--(c) are verified in exactly
the same way as in the proof of Theorem~\ref{thm:CLT} (with Corollary~\ref{cor:Burk} replaced by Corollary~\ref{cor:Burk_exp}).
The remainder of the proof of step~(ii) is identical to Theorem~\ref{thm:CLT}, again using~\cite[Theorem~6.1]{KM16} to pass the iterated WIP from $Y$ to $Y^H$ and hence $M$.
\end{proof}

\begin{thm} \label{thm:moment_exp}
Define $v_t$, $S_t$ as in~\eqref{eq:vS}.
\begin{itemize}
\item[(a)]
If $H\in L^p(Y)$ for some $p\ge2$, then there is a constant $C>0$ such that $|v_t|_{p-1}\le Ct^{1/2}\|v\|_\eta$ for all $v\in C^\eta_0(M,\R^d)$, $t>0$.
\item[(b)]
If $H\in L^p(Y)$ for some $p\ge4$, then there is a constant $C>0$ such that $|S_t|_{(p-1)/2}\le Ct\|v\|_\eta^2$ for all $v\in C^\eta_0(M,\R^d)$, $t>0$.
\end{itemize}
\end{thm}

\begin{proof}
The proof is identical to that of Theorem~\ref{thm:moment}.
\end{proof}

\begin{rmk}  \label{rmk:rearrange}
(i) Again, it follows that the formulas for $\Sigma$ and $E$ in Theorem~\ref{thm:WIP}
hold provided $p\ge4$, and that $\det\Sigma>0$ for $v$ lying outside a closed subspace of infinite codimension.
\\ (ii)
Under appropriate hypotheses on the flow, it is possible to incorporate the rearrangement idea from Section~\ref{sec:rearrange} and hence to obtain the improved moment estimates in Theorem~\ref{thm:momentDelta}.
We omit the tedious details.
\end{rmk}

\section{Nonuniformly hyperbolic flows with horseshoe structure}
\label{sec:prod}

In this section, we relax the assumption~\eqref{eq:Wsexp} about exponential contraction along stable leaves under the flow.
Instead we assume that $F:Y\to Y$ is a nonuniformly hyperbolic horseshoe
as in~\cite{Young98}.  We begin by describing the basic set up, focusing on the parts needed for understanding what follows and referring to~\cite{Young98} for further details.

We assume the structure from Section~\ref{sec:stand}, including the measurable partition $\cW^s$ of $Y$ into stable leaves.  Further, we suppose that there is a second measurable partition $\cW^u$ of $Y$ consisting of
{\em unstable leaves}.

We assume that there are constants $C>0$, $\gamma_0\in(0,1)$, $\eta\in(0,1]$,
such that for all $y\in Y$, $y'\in W^s_y$, $t\in[0,H(y)]\cap[0,H(y')]$,
\begin{align} \label{eq:Hd}
& |H(y)-H(y')|  \le C({\SMALL\inf_{Y_j}}H)d(y,y')^\eta, \qquad
d(T_ty,T_ty')  \le Cd(y,y')^\eta;
\end{align}
and for all $y,y'\in Y_j$, $j\ge1$, $y'\in W^u_y$, $t\in[0,H(y)]\cap[0,H(y')]$,
\begin{align} \label{eq:Hdu}
& |H(y)-H(y')|  \le C({\SMALL\inf_{Y_j}}H)\gamma_0^{s(y,y')}, \qquad
d(T_ty,T_ty')  \le C\gamma_0^{s(y,y')}.
\end{align}
Also, we assume that for all $y,y'\in Y_j$, $j\ge1$, $n\ge0$,
\begin{alignat}{2}
\label{eq:Ws}
& d(F^ny,F^ny')  \le C\gamma_0^n
&& \quad\text{for $y'\in W^s_y$},
\\ & d(y,y')  \le C\gamma_0^{s(y,y')}
&& \quad\text{for $y'\in W^u_y$}.
\label{eq:Wu}
\end{alignat}

Define the induced observable $V(y)=\int_0^{H(y)}v(T_uy)\,du$ on $Y$.

\begin{prop} \label{prop:VY_prod}
Let $v\in C^\eta(M,\R^d)$, and set $\gamma_1=\gamma_0^{\eta}$.  Then
$|V|\le |v|_\infty H$ and
there is a constant $C>0$ such that
$|V(y)-V(y')|\le C\|v\|_\eta ({\SMALL\inf}_{Y_j}H) d(y,y')^{\eta^2}$
for $y\in Y$, $y'\in W^s_y$; while 
$|V(y)-V(y')|\le C\|v\|_\eta ({\SMALL\inf}_{Y_j}H) \gamma_1^{s(y,y')}$
for $y,y'\in Y_j$, $j\ge1$, $y'\in W^u_y$.
\end{prop}
\begin{proof} The first estimate is immediate.
Let $y'\in W^s_y$ and suppose without loss that $H(y)\ge H(y')$.  Then
\begin{align*}
|& V(y)  - V(y')|
 \le \int_{H(y')}^{H(y)}|v(T_uy)|\,du+\int_0^{H(y')}|v(T_uy)-v(T_uy')|\,du \\
& \le |v|_\infty(H(y)-H(y'))+|v|_\eta H(y')\max_{u\in[0,H(y')]}d(T_uy,T_uy')^\eta
 \ll \|v\|_\eta({\SMALL\inf}_{Y_j}H)d(y,y')^{\eta^2},
\end{align*}
by~\eqref{eq:Hd}. The last estimate follows by the same argument using~\eqref{eq:Hdu} instead of~\eqref{eq:Hd}.~
\end{proof}

\subsection{Martingale-coboundary decomposition}

We apply a Gordin type argument~\cite{Gordin69} to obtain an $L^2$ martingale-coboundary decomposition.
Let $\overline{\cB}$ denote the underlying $\sigma$-algebra on $\bY$
and let $\cB=\bpi^{-1}\overline{\cB}$.
By Proposition~\ref{prop:VY_prod},
$|\E(V|\cB)|_1\le |V|_1\le |v|_\infty|H|_1$.
Since $\E(V|\cB)$ is constant along stable leaves, we can write
$\E(V|\cB)=\bV\circ \bpi$ where $\bV\in L^1(\bY)$.

Let $P$ be the transfer operator for the Gibbs-Markov map $\bF:\bY\to \bY$.

\begin{prop}  \label{prop:m_prod}
Suppose that $H\in L^p(Y)$ for some $p\ge1$.  There is a constant $C>0$ such that
\begin{align*}
& \SMALL  \sum_{n\ge0} \big|\E(V\circ F^n|\cB)-V\circ F^n\big|_{L^p(Y)}\le C\|v\|_\eta, \qquad
 \sum_{n\ge 1} |P^n\bV|_{L^\infty(\bY)}\le C\|v\|_\eta,
\end{align*}
for all $v\in C^\eta_0(M,\R^d)$.
\end{prop}

\begin{proof}
Set $\gamma_2=\gamma_0^{\eta^2}$.
Let
$y,y'\in Y$ with $y'\in W^s_y$.  Then
\[
|V(F^ny)-V(F^ny')|\ll \|v\|_\eta\, H(F^ny) d(F^ny,F^ny')^{\eta^2}
\ll \gamma_2^n\|v\|_\eta\, H(F^ny),
\]
 by Proposition~\ref{prop:VY_prod} and~\eqref{eq:Ws}.
Hence
$|\E(V\circ F^n|\cB)(y)-V(F^ny)|
\ll \gamma_2^n\|v\|_\eta\, H(F^ny)$ and
\[
|\E(V\circ F^n|\cB)-V\circ F^n|_p\ll \gamma_2^n \|v\|_\eta\, |H|_p,
\]
so the first sum converges.

Now define $g\in L^1(Y)$ and $\bar g\in L^1(\bY)$ by setting
$g(y)=\inf_{Y_j}H$ for all $y\in Y_j$ and writing $g=\bar g\circ\bpi$.
By Proposition~\ref{prop:VY_prod}, we have
$|V|\ll |v|_\infty\, g$ and
$|V(y)-V(y')|\ll \|v\|_\eta\,g(y)\gamma_1^{s(y,y')}$ for all $y,y'\in Y_j$, $j\ge1$, with $y'\in W^u_y$.
It follows that $|\bV|\ll |v|_\infty\,\bar g$.
Arguing as in~\cite[Sublemma, p.~612]{Young98} (see~\cite[Lemma~5.4]{BalintGouezel06}),
$|\bV(y)-\bV(y')|\ll \|v\|_\eta\, \bar g(y)\gamma_1^{s(y,y')}$
for $y,y'\in\bY\!_j$
(cf.\ \cite[Lemma~5.4]{BalintGouezel06}).
By Proposition~\ref{prop:GM}, there exists $\gamma\in(0,1)$ such that
$|P^n\bV|_\infty\ll \gamma^n \|P\bV\|_{\gamma_1}\ll\gamma^n\|v\|_\eta$ for $n\ge1$.
Hence $\sum_{n\ge1}|P^n\bV|_\infty<\infty$, completing the proof.~
\end{proof}

\begin{cor} \label{cor:m_prod}
Let $v\in C^\eta_0(M,\R^d)$ and suppose that $H\in L^p(Y)$ for
some $p\ge1$.
There exists
$\bm\in L^p(\bY,\R^d)$, $\chi\in L^p(Y,\R^d)$ such that
$V=\bm\circ\bpi+\chi\circ F-\chi$ and $\bm\in\ker P$.

Moreover, there is a constant $C>0$ such that
$|\bm|_p\le C\|v\|_\eta$ and
$|\chi|_p\le C\|v\|_\eta$ for all $v\in C^\eta_0(M,\R^d)$.
\end{cor}

\begin{proof}  Define $g_n=\E(V\circ F^n|\cB)$ and
$h_n=(P^n\bV)\circ\bpi$.
By Proposition~\ref{prop:m_prod},
\[
\SMALL \chi=\sum_{n\ge 0} (g_n-V\circ F^n)
+\sum_{n\ge 1} h_n
\]
converges in $L^p(Y)$.
Define
$m=V+\chi-\chi\circ F\in L^p(Y)$.
We claim that (i) $m$ is $\cB$-measurable and
(ii) $\E(m|F^{-1}\cB)=0$.
Property~(i) ensures that $m=\bm\circ \bpi$ where
$\bm:\bY\to\R^d$.
By property~(ii),
\begin{align*}
(UP\bm)\circ\bpi & = \E(\bm|\bF^{-1}\overline{\cB})\circ\bpi
=\E(\bm\circ\bpi|\bpi^{-1}\bF^{-1}\overline{\cB})
=\E(m|F^{-1}\cB)=0,
\end{align*}
where $Uv=v\circ \bF$.
Hence $UP\bm=0$ and so $\bm\in\ker P$.

It remains to verify the claim.  We have
\begin{align} \label{eq:m}
\SMALL m=\sum_{n\ge0} g_n-\sum_{n\ge0} g_n\circ F
+\sum_{n\ge1} h_n-\sum_{n\ge1} h_n\circ F.
\end{align}
Clearly, $g_n=\E(V\circ F^n|\cB)$ is $\cB$-measurable.  Also,
$P^n\bV$ is $\overline{\cB}$-measurable,
so $h_n=(P^n\bV)\circ\bpi$ is measurable with respect to $\bpi^{-1}\overline{\cB}=\cB$.
Next, $g_n\circ F$ is measurable with respect to $F^{-1}\cB\subset \cB$ and similarly for $h_n\circ F$.  Part (i) of the claim follows.

For part~(ii) of the claim, note that
$g_n\circ F=\E(V\circ F^n|\cB)\circ F=
\E(V\circ F^{n+1}|F^{-1}\cB)$.  Hence
\[
\E(g_n\circ F|F^{-1}\cB)=\E(V\circ F^{n+1}|F^{-1}\cB)=\E(g_{n+1}|F^{-1}\cB).
\]
Also,
\begin{align*}
h_n\circ F=(P^n\bV)\circ \bF\circ\bpi
=(UP^n\bV)\circ\bpi = \E(P^{n-1}\bV|\bF^{-1}\overline{\cB})\circ\bpi
= \E(h_{n-1}|F^{-1}\cB),
\end{align*}
so $\E(h_n\circ F|F^{-1}\cB)=\E(h_{n-1}|F^{-1}\cB)$.
Substituting into~\eqref{eq:m},
$\E(m|F^{-1}\cB)=
\E(g_0|F^{-1}\cB)- \E(h_0|F^{-1}\cB)$.
Finally, $h_0=\bV\circ\bpi=\E(V|\cB)=g_0$, so
$\E(m|F^{-1}\cB)=0$ as required.
\end{proof}

\begin{cor} \label{cor:Burk_prod}
Suppose that $H\in L^p(Y)$ for some $p\ge2$.
There exists $C>0$ such that
$\big|\max_{1\le k\le n}|V_k|\big|_p\le Cn^{1/2}\|v\|_\eta$
for all $v\in C^\eta_0(M,\R^d)$, $n\ge1$.
\end{cor}

\begin{proof}  As in the proof of Corollary~\ref{cor:Burk},
$\big|\max_{1\le k\le n}|(\bm\circ\bpi)_k|\big|_{L^p(Y)}=
\big|\max_{1\le k\le n}|\bm_k|\big|_{L^p(\bY)}\ll n^{1/2}|\bm|_{L^p(\bY)}
\ll n^{1/2}\|v\|_\eta$.
Also, 
$\int_Y \max_{k\le n}|\chi|^p\circ F^k\,d\mu_Y
\le \int_Y \sum_{k\le n}|\chi|^p\circ F^k\,d\mu_Y
= n |\chi|_p^p\ll n\|v\|_\eta^p$.
Since $p\ge2$, the result follows.
\end{proof}

\subsection{Central limit theorems and moment estimates}

We continue to assume the set up from the beginning of the section.

\begin{thm}   \label{thm:CLT_prod}
Suppose that $H\in L^2(Y)$.
Let $v\in C^\eta_0(M,\R^d)$.
Then the statement and conclusion of Theorem~\ref{thm:CLT} holds.
That is, $(W_n,\BBW_n)\to_w (W,\BBW)$ in $C([0,\infty),\R^d\times\R^{d\times d})$
on $(M,\mu_M)$, where $W_n$, $\BBW_n$, $W$ and $\BBW$ are defined as in
Theorem~\ref{thm:CLT}.
\end{thm}

\begin{proof}
Define $W_n^Y$, $\BBW_n^Y$ on $(Y,\mu_Y)$ as in the proof of Theorem~\ref{thm:CLT}.
In place of Proposition~\ref{prop:m} and~\cite[Theorem~4.3]{KM16}, we use
Corollary~\ref{cor:m_prod} and~\cite[Theorem~5.2]{KM16} to deduce that
$(W_n^Y,\BBW_n^Y)\to(W^Y,\BBW^Y)$ in $D([0,\infty),\R^d\times\R^{d\times d})$
where $W^Y$ and $\BBW^Y$ are as in step~(i) of the proof of Theorem~\ref{thm:CLT}.

Claims (a)--(c) are verified in exactly
the same way as in the proof of Theorem~\ref{thm:CLT} (with Corollary~\ref{cor:Burk} replaced by Corollary~\ref{cor:Burk_prod}).
The remainder of the proof of step~(ii) is identical to Theorem~\ref{thm:CLT}, again using~\cite[Theorem~6.1]{KM16} to pass the iterated WIP from $Y$ to $Y^H$ and hence $M$.
\end{proof}

\begin{thm} \label{thm:moment_prod}
Define $v_t$, $S_t$ as in~\eqref{eq:vS}.
Suppose that $\inf H>0$.
\begin{itemize}
\item[(a)]
If $H\in L^p(Y)$ for some $p\ge2$, then there is a constant $C>0$ such that $|v_t|_{p-1}\le Ct^{1/2}\|v\|_\eta$ for all $v\in C^\eta_0(M,\R^d)$, $t>0$.
\item[(b)]
If $H\in L^p(Y)$ for some $p\ge4$, then there is a constant $C>0$ such that $|S_t|_{(p-1)/2}\le Ct\|v\|_\eta^2$ for all $v\in C^\eta_0(M,\R^d)$, $t>0$.
\end{itemize}
\end{thm}

\begin{proof}
The proof is almost identical to that of Theorem~\ref{thm:moment}.  The only difference is that $\chi\in L^p(Y,\R^d)$ (instead of $\chi\in L^\infty(Y,\R^d)$) but it is easily checked that the required estimates go through.
\end{proof}

\section{Singular hyperbolic attractors}
\label{sec:Lorenz}

In this section, we verify that our main results apply to codimension two singular hyperbolic attractors, as described in the introduction.

\begin{lemma} \label{lem:SH}  Codimension two singular hyperbolic attractors (Definition~\ref{def:SH})  satisfy the hypotheses in Section~\ref{sec:exp} with $\beta$ arbitrarily large.
\end{lemma}

\begin{proof}
We apply the construction in~\cite{AraujoMsub}, representing the singular hyperbolic attractor $\Lambda$ as a suspension
$\Lambda=X^h$, $X=Y^\tau$, where
$\mu_X(h>t)=O(e^{-ct})$ for some $c>0$ and $\mu_Y(\tau>n)=O(n^{-\beta})$
with $\beta$ arbitrarily large.
By Remark~\ref{rmk:beta}(c), the induced roof function $H=h_\tau$ satisfies
$\mu_Y(H>t)=O(t^{-\beta})$ with $\beta$ arbitrarily large.
(In~\cite{AraujoMsub}, our $h$ and $\tau$ are denoted $\tau$ and $\rho$ respectively.)

Singular hyperbolic attractors have the property that there is a H\"older stable foliation for the flow~\cite[Theorem~6.2]{AraujoMsub} and hence the projection
$\bpi:Y\to\bY$ extends to $\bpi:X\to\bX$.
As in~\cite[Section~9]{AraujoMsub}, we can suppose without loss that
$h$ and $H$ are constant along stable leaves for the flow and hence
are well-defined on $\bX$ and $\bY$ respectively.
Define the quotient Poincar\'e map $\bar f:\bX\to\bX$.
By~\cite[Lemma~8.4]{AraujoMsub}, the induced quotient map
$\bF=\bar f^\tau:\bY\to\bY$ is Gibbs-Markov.

By~\cite[Proposition~8.6]{AraujoMsub},
\begin{align*}
|H(y)-H(y')|\le
\SMALL \sum_{\ell=0}^{\tau(y)-1}|h(\bar f^\ell y)-h(\bar f^\ell y')|
\ll |\bF y-\bF y'|^\eps
\quad\text{for $y,y'\in \bY\!_j$, $j\ge1$.}
\end{align*}
Hence $|H(y)-H(y')|\ll \gamma_0^{s(y,y')}$ for $y,y'\in \bY\!_j$, $j\ge1$.
It follows immediately that
$|H(y)-H(y')|\ll \gamma_0^{s(y,y')}$ for $y,y'\in Y_j$, $j\ge1$,
so condition~\eqref{eq:inf_exp} is satisfied.
Also, since the stable foliation is H\"older,
\begin{align} \label{eq:AFLV}
|H(y)-H(y')|
\ll |F y-F y'|^{\eps^2}
\quad\text{for $y,y'\in Y_j$, $j\ge1$.}
\end{align}
Condition~\eqref{eq:holder_exp} is immediate from smoothness of the flow (with $\eta=1$).
Singular hyperbolic flows contract exponentially along stable leaves by definition, so~\eqref{eq:Wsexp} holds.
Finally we verify~\eqref{eq:hY}.  There are three steps.
\\[.75ex]
{\bf Step 1.}
We show that $|T_u x-T_u x'|\ll |x-x'|^\eps+|f x-fx'|^\eps$
for all $x,x'\in X$, $u\in[0,h(x)]$.
The first return time $h$ is either bounded or of the form $h=h_1+h_2$ where $h_1$ is bounded and $h_2$ is the first hit map in a neighbourhood of a steady-state for the flow.  The claim is clear for bounded return times, so it remains to consider  the first hit map $h_2$.  Again, the claim is immediate if the flow near the equilibrium is linear.
By a H\"older version of the Hartman-Grobman theorem, linearity always holds up to a bi-H\"older change of coordinates completing the proof of Step 1.
\\[.75ex]
{\bf Step 2.}
We show that $|T_t y-T_t y'|\ll |F y-Fy'|^{\eps^2}$
for all $y,y'\in \hY_j$, $j\ge1$,
$t\in[0,H(y)]\cap [0,H(y')]$.
There exists $0\le \ell,\ell'\le \tau(y)-1$ such that
$t\in[h_\ell(y),h_{\ell+1}(y)]
\cap [h_{\ell'}(y'),h_{\ell'+1}(y')]$.  Without loss $\ell\le \ell'$.
Then
$t=h_\ell(y)+u=h_\ell(y')+u'$ where $u\in[0,h(f^\ell(y)]$ and
$u'=t-h_\ell(y')\ge t-h_{\ell'}(y')\ge0$.
By Step 1 and smoothness of the flow,
\begin{align*}
|T_ty-T_ty'| & =|T_u f^\ell y - T_{u'}f^\ell y'|
\le |T_u f^\ell y - T_uf^\ell y'|
+ |T_u f^\ell y' - T_{u'}f^\ell y'|
\\ & \ll |f^\ell y-f^\ell y'|^\eps+|f^{\ell+1} y-f^{\ell+1} y'|^\eps +|u-u'|
\\ & =|f^\ell y-f^\ell y'|^\eps+ |f^{\ell+1} y-f^{\ell+1} y'|^\eps +|h_\ell(y)-h_\ell(y')|.
\end{align*}
By~\eqref{eq:AFLV},
$|h_\ell(y)-h_\ell(y')|\ll |F y -F y'|^{\eps^2}$.
By the construction in~\cite{AraujoMsub} which is based on~\cite{AFLV11,AlvesLuzzattoPinheiro05},
$|f^\ell y-f^\ell y'|\ll |\bF y-\bF y'|^\eps\ll |Fy-Fy'|^{\eps^2}$ for $\ell\le \tau(y)$.
Hence, $|T_ty-T_ty'|\ll |F y-F y'|^{\eps^2}$ as required.
\\[.75ex]
{\bf Step 3.}  We complete the verification of~\eqref{eq:hY}.
Let $y,y'\in\hY\cap Y_j$.
Let $n=J(t)(y)$, $n'=J(t)(y')$ where $J(t)$ is as in~\eqref{eq:hY},
and suppose without loss that $n\le n'$ and $n\le s(y,y')$.
Write
$T_ty=T_uF^ny$, $T_ty'=T_{u'}F^ny'$
where $u\in[0,H(F^ny)]$ and $u'\ge0$.
By Step 2 and existence of an unstable cone field~\cite[Proposition~4.2]{AraujoMsub},
$|T_uF^ny-T_uF^ny'|\ll |F^{n+1}y-F^{n+1}y'|^\eps$.
Using this and smoothness of the flow,
\begin{align*}
|T_ty-T_ty'| & =|T_u F^n y - T_{u'}F^n y'|
\le |T_u F^n y - T_uF^n y'|
+ |T_u F^n y' - T_{u'}F^n y'|
\\ & \ll |F^{n+1} y-F^{n+1} y'|^\eps +|H_n(y)-H_n(y')|.
\end{align*}
By~\eqref{eq:AFLV}, $|T_ty-T_ty'|\ll \sum_{j=1}^{n+1}|F^jy-F^jy'|^{\eps^2}$.
By~\cite[Proposition~4.2]{AraujoMsub}, there exists $\lambda_1\in(0,1)$ such that
$|\bar F^{s(y,y')}y-\bar F^{s(y,y')}y'|
\ge \lambda_1^{-(s(y,y')-j}|F^jy-F^jy'|$, so
$|T_ty-T_ty'|\ll \gamma_0^{s(y,y')-n}=\gamma_0^{s(y,y')-J(t)(y)}$ as required
with $\gamma_0=\lambda_1^{\eps^2}$.
\end{proof}

\appendix

\section{Tails for induced roof functions}

For completeness, we state and prove the following standard result.

\begin{prop} \label{prop:H}
$\mu_Y(H>t)\le \mu_Y(\tau >n)+ \bar\tau \mu_X(h>t/n)$
for all $n\ge1$, $t>1$.
\end{prop}

\begin{proof}
Recall that 
$\mu^\tau=c^{-1}(\mu_Y\times{\rm counting})$,
$c=\bar\tau$.
Also, $\mu_X=\pi_*\mu^\tau$ where $\pi(y,\ell)=f^\ell y$.

Let $h^*(y)=\max\{h(y),\,h(fy),\,\dots,h(f^{\tau(y)-1}y)\}$ and
define
\[
\ell(y)=\min\{\ell\in\{0,\dots,\tau(y)-1\}:h(f^\ell y)=h^*(y)\}.
\]
Then
\begin{align*}
\mu_Y \{ &    H>t,\,\tau\le n\}
\le \mu_Y(h^*>t/n)=
  \mu_Y\{y\in Y:h(f^{\ell(y)}y) > t/n\}
\\ & = c\mu^\tau\{(y,0)\in Y^\tau:h(f^{\ell(y)}y) > t/n\}
= c \mu^\tau\{(y,\ell(y))\in  Y^\tau:h(f^{\ell(y)}y) > t/n\}
\\ & = c\mu^\tau\{(y,\ell(y))\in  Y^\tau:h\circ\pi(y,\ell(y)) > t/n\}
\le  c\mu^\tau\{p\in  Y^\tau:h\circ\pi(p) > t/n\}
\\ & =  c\mu_X\{x\in X:h(x) > t/n\}.
\end{align*}
Hence $\mu_Y(H>t)\le \mu_Y(\tau >n)+\mu_Y(H>t,\,\tau\le n)
\le \mu_Y(\tau >n)+c\mu_X(h>t/n)$ as required.
\end{proof}

\begin{rmk}  \label{rmk:beta}
For completeness, we list some special cases of Proposition~\ref{prop:H}:
\begin{itemize}

\parskip = -2pt
\item[(a)]
If $\mu_X(h>t)=O(t^{-\beta})$ and $\mu_Y(\tau >n)$ decays exponentially,
then $\mu_Y(H>t)=O((\log t)^{\beta} t^{-\beta})$.
(Choose $n=C\log t$ with $C$ large.)
\item[(b)]
If $\mu_X(h>t)=O(e^{-ct})$ and $\mu_Y(\tau >n)=O(e^{-dn})$,
then $\mu_Y(H>t)=O(e^{-(cdt)^\frac12})$.  (Choose $n=\sqrt{ct/d}$.)
\item[(c)]
If $\mu_X(h>t)=O(e^{-ct})$ and $\mu_Y(\tau >n)=O(n^{-\beta})$,
then $\mu_Y(H>t)=O((\log t)^\beta t^{-\beta})$.  (Choose $n=at(\log t)^{-1}$ with $a$ small.)
\item[(d)]
If $\mu_X(h>t)=O(t^{-\beta_1})$ and $\mu_Y(\tau >n)=O(n^{-\beta_2})$,
then $\mu_Y(H>t)=O(t^{-\beta})$ where
$\BIG\beta=\frac{\beta_1\beta_2}{\beta_1+\beta_2}$.
(Choose $n=t^{\beta_1/(\beta_1+\beta_2)}$.)
\end{itemize}
\end{rmk}

\paragraph{Acknowledgements}
The research of PB was supported in part by
Hungarian National Foundation for Scientific Research (NKFIH OTKA)
grants K104745 and K123782.
The research of IM was supported in part by a
European Advanced Grant {\em StochExtHomog} (ERC AdG 320977).
We are grateful to the referee for very helpful suggestions.

\def\polhk#1{\setbox0=\hbox{#1}{\ooalign{\hidewidth
  \lower1.5ex\hbox{`}\hidewidth\crcr\unhbox0}}}

\end{document}